\documentclass[11pt, english]{article}

\usepackage[english]{babel}
\usepackage{textgreek}
\usepackage[thms]{packages}
\usepackage{thm-restate}
\usepackage{todonotes}
\usepackage{bbm}
\usepackage{dsfont}
\usepackage{bm}
\usepackage[singlelinecheck=false, font=small]{caption}
\usepackage{tikz-cd}

\usepackage{authblk}

\usepackage{amsmath, mathrsfs}

\usepackage{xcolor}

\usepackage{aligned-overset}

\usepackage{changepage}

\usepackage{wasysym}

\usepackage{arrayjobx}
\usepackage{trimspaces}

\makeatletter
\def\trimspace#1{\trim@spaces@in{#1}}
\makeatother

\usepackage[style=trad-alpha,backend=biber, url=false, isbn=false]{biblatex}
\usepackage[unicode]{hyperref}
\usepackage{orcidlink}
\usepackage[capitalise]{cleveref}
\numberwithin{equation}{section}

\usepackage[intoc]{nomencl}
\makenomenclature

\renewcommand{\d}[1]{\ensuremath{\;\operatorname{d}\!{#1}}}

\title{A $\Lambda$-Fleming-Viot type model with intrinsically varying population size}
\author[1]{Julian Kern \orcidlink{0000-0002-8231-0736}}
\author[2]{Bastian Wiederhold \orcidlink{0009-0000-4985-9076}}
\affil[1]{Weierstraß Institute for Applied Analysis and Stochastics, Berlin, Germany}
\affil[2]{Department of Statistics, University of Oxford, Oxford, UK}

\begin{document}

\maketitle

\begin{abstract}
We propose an extension of the classical $\Lambda$-Fleming-Viot model to intrinsically varying population sizes. 
During events, instead of replacing a proportion of the population, a random mass dies and a, possibly different, random mass of new individuals is added. 
The model can also incorporate a drift term, representing infinitesimally small, but frequent events. 
We investigate elementary properties of the model, analyse its relation to the $\Lambda$-Fleming-Viot model and describe a duality relationship.
Through the lookdown framework, we provide a forward-in-time analysis of fixation and coming down from infinity.
Furthermore, we present a new duality argument allowing one to deduce well-posedness of the measure-valued process without the necessity of proving uniqueness of the associated lookdown martingale problem.
\end{abstract}


\section{Introduction}
\paragraph{Genetics and Ecology} 
Early mathematical models of evolution could often be classified as belonging either to genetics or population dynamics. 
The former focusses on the development of the genetic information carried by the population through biological mechanisms, most prominently mutation, selection, recombination and genetic drift. 
These effects already result in a high amount of stochasticity. 
In order to handle the models, compromises are made in other areas such as the interplay with the environment or other species. 
These, in turn, are the strongholds of population dynamics, which often neglect the genetic code of individuals.

It has become a major endeavour of population modelling to weaken this separation.
A common characteristic of models at the intersection of ecology and genetics is to incorporate a form of varying population size required for the implementation of ecological factors such as external catastrophes, selective sweeps, interspecies relationships or competition for a fluctuating abundance of resources. 
Conceptually, the population size can vary extrinsically or intrinsically.
In the first case, the population size enters as a parameter into the development of the population and is driven by a randomness independent of the remaining stochasticity. 
This lends itself to modelling external influences such as catastrophes or the loss of habitat, see e.g.~\cite{KK03,GMS22,lambdadormancy}.
Whereas the population size is assumed to be only slowly varying in \cite{KK03}, it is taken to undergo drastic changes (so-called \emph{bottlenecks}) in \cite{GMS22,lambdadormancy}, leading to vastly different genealogies. 

In contrast, models with intrinsically varying population size inextricably link the reproductive dynamics with the population size.
The recent works \parencite{EK19} and \parencite{EKLRT23} have emphasized the potential of the lookdown framework in analysing intrinsically varying models, see also \parencite{DK96,DK99,DK99b}.
The former provides a toolbox for population models and constructs a spatial $\Lambda$-Fleming-Viot model with varying population size, which has been a major motivation for this work. 
The latter analyses a locally regulated population with a two-step reproduction mechanism with a juvenile and adult stage, which allows for a wide range of population profile patterns to appear.  

\paragraph{A varying size \bm{$\Lambda$}-Fleming-Viot process}
With this paper, we wish to further contribute to the understanding of the implications of intrinsically varying population size by considering a toy model based on a variation of the classical $\Lambda$-Fleming-Viot paradigm.
More precisely, the model will evolve through a series of reproduction events during which a random proportion ${z}_d$ of the population is killed and a random mass ${z}_b$ of offspring with a type chosen uniformly from the population prior to the event is added.
To ensure persistence of the population, we will allow for an appropriate drift in the population size.
Although to our knowledge, this specific model has not yet been considered, it fits into the general framework from \cite{DK99}. A similar model in the spatial context, but under much more restrictive assumptions can be found in \cite[Section 4.2]{EK19}.
Our model should also be compared to the model in \cite{Sch03}, where individuals can be parental to large families, but the population size is reduced to a fixed quantity after reproduction.

The model is not a \emph{generalisation} of the $\Lambda$-Fleming-Viot model in a strict sense: although it covers $\Lambda$-Fleming-Viot processes for a range of impact measures $\Lambda$, not all are included.
Instead, it should be viewed as the natural variant for which neither the offspring distribution nor the death proportion is restricted.
This is complemented by the fact that we can approximate any classical $\Lambda$-Fleming-Viot process arbitrarily well.
Despite this relatedness with classical models, the interplay between varying population size and reproduction events will necessitate additional care in the analysis.
Even though duality is still available in a restricted sense, the dual coalescent is time inhomogeneous and not suitable for analysis.
Instead of studying fixation and coming down from infinity backwards in time, we develop a \emph{forward-in-time} approach to analyse the genealogy.

The major tool in this analysis is the lookdown construction from \cite{DK99} which allows one to represent the measure-valued model as a countable particle system, where each individual has a unique label or level in $\mathbb{N}$.
The main advantage of lookdown constructions is the fact that offspring always ``look down" to their parent: from all individuals participating in a reproduction event, the one with the lowest level is parental.
This incorporates the corresponding coalescent backwards in time in a consistent way and allows one to trace lines of descent forward in time, independently of their changing type.

\paragraph{Outline} 
\Cref{sec:model} contains the definition of the model, presents the main results and closes with a discussion of the model choice.
In \Cref{sec:populationsize} we will start our analysis by collecting properties of the population size process. 
\Cref{sec:closedness} shows the closedness of the class of processes under a natural convergence condition. 
In \Cref{sec:relationlambdafvprocess} we will see that any $\Lambda$-Fleming-Viot process can be recovered as a limit of our processes.
\Cref{sec:well-posedness} recalls the lookdown construction of \cite{DK99} and provides a proof through the Markov Mapping Theorem.
The corresponding dual process is identified in \Cref{sec:duality}, where we also examine the genealogy.
Code for simulation is provided at \url{https://github.com/mushunrek/GammaPiModel.jl}.


\section{Model and main results} \label{sec:model}

\subsection{The $($\texorpdfstring{$\bm{\gamma}$}{\gamma}$,\Pi)$-Fleming-Viot process} \label{ssec:definition_model}
Throughout this work, we consider measure-valued populations taking values in the space $\mathcal{M}_F(\mathbb{K})$ of finite measures on a compact Polish type space $\mathbb{K}$, usually $\mathbb{K} = \{a,A\}$ or $\mathbb{K} = [0,1]$. 
Denote by $\mathcal{M}_1(\mathbb{K})\subseteq \mathcal{M}_F(\mathbb{K})$ the set of probability measures on $\mathbb{K}$.

\begin{defin}[Events and Characteristics]
In the context of deaths and births, we will always denote pairs $\bm z, \bm\gamma\in\mathbb{R}^2$ by bold letters and their entries by $(z_d,z_b)$ and $(\gamma_d,\gamma_b)$.
\begin{enumerate}
	\item We denote by
	\[
	\mathcal{Z} := \Big( [0,1]\times [0,+\infty)\Big) \setminus \{(0,0), (1,0)\}
	\]
	the \emph{event space}. Throughout this paper, its elements will be referred to as \emph{events} and be denoted by $\bm z =: (z_d, z_b)$.
	Here, $z_d$ denotes the proportion of individuals to be killed, and $z_b$ the mass of offspring added during the event $\bm z$.
	
	\item We define the space $\mathfrak{C}$ of \emph{characteristics} to contain exactly the pairs $(\bm\gamma,\Pi)\in \mathbb{R}^2\times \mathcal{M}(\mathcal{Z})$ satisfying
	\begin{enumerate}[label=\roman*)]
		\item $\bm\gamma \geq \bm 0$,
		\item the subordinator condition
		\begin{equation}\label{eq:subordinator_condition}
		\int_\mathcal{Z} \bm z \d\Pi(\bm z) < +\infty,
		\end{equation}
		\item and the balance condition
		\begin{equation}\label{eq:balance_condition}
			\gamma_d + \int_\mathcal{Z} z_d \d\Pi(\bm z) = \gamma_b + \int_\mathcal{Z} z_b\d\Pi(\bm z).
		\end{equation}
	\end{enumerate}
	In the above, $\mathcal{M}(\mathcal{Z})$ denotes the set of $\sigma$-finite measures on $\mathcal{Z}$.
	The inequalities are to be understood entrywise.
	In analogy to $\bm z$, we will write $\Pi_d$ and $\Pi_b$ for the marginals of $\Pi$.
\end{enumerate}
\end{defin}
The events $(0,0)$ and $(1,0)$ are excluded to avoid trivialities: the first does not change the population, whereas the second kills off all individuals without replacing any.
The second condition \eqref{eq:subordinator_condition} is necessary for the process to be well-defined. It is referred to as \emph{subordinator condition} as it implies together with $\bm\gamma\geq \bm 0$ that $(\bm\gamma,\Pi)\in\mathfrak{C}$ defines a two-dimensional subordinator with drift $\bm\gamma$ and jump measure $\Pi$.

\begin{defin}[The $(\bm\gamma,\Pi)$-Fleming-Viot process]
Let $(\bm\gamma,\Pi)\in\mathfrak{C}$ be a characteristic.
Define the operator
\begin{equation}\label{eq:operator_for_sigma}
\begin{array}{ll}\mathcal{L}^{\bm\gamma,\Pi} \Phi(\sigma)\!\!\!\!\! &= 
\big( \gamma_b - \gamma_d \sigma(\mathbb{K})\big) D\Phi\left(\sigma, \dfrac{\sigma}{\sigma(\mathbb{K})}\right)\\[10pt]
&\qquad \displaystyle+ \int_\mathbb{K}\int_\mathcal{Z} \left[ \Phi\Big( (1-z_d)\sigma + z_b\delta_\kappa\Big) - \Phi(\sigma)\right] \d\Pi(\bm z)\dfrac{\mathrm{d}\sigma(\kappa)}{\sigma(\mathbb{K})}
\end{array}
\end{equation}
for all bounded Fréchet-differentiable $\Phi\in C_b^1\mathcal{M}_F(\mathbb{K})$ with bounded derivative.
Any $\mathcal{M}_F(\mathbb{K})$-valued solution $\sigma$ to the martingale problem for $\mathcal{L}^{\bm\gamma,\Pi}$ is called a \emph{$(\bm\gamma,\Pi)$-Fleming-Viot process}.
\end{defin}
In \Cref{subsec:discussion} we give a detailed motivation of the model choice. Any $\sigma\in\mathcal{M}_F(\mathbb{K})\setminus \{0\}$ can be uniquely decomposed into $\sigma = N\rho$, where $N := \sigma(\mathbb{K})$ is the \emph{population size} and the probability measure $\rho\in \mathcal{M}_1(\mathbb{K})$ is the \emph{type distribution}.
If $\sigma$ is a $(\bm\gamma,\Pi)$-Fleming-Viot process, then $N$ is solution to the martingale problem for 
\begin{equation}\label{eq:operator_for_N}
\mathcal{L}^{\bm\gamma,\Pi} f(N) = \big( \gamma_b - \gamma_d N\big) f'(N) + \int_\mathcal{Z} \left[ f\Big( (1-z_d)N + z_b\Big) - f(N)\right] \d\Pi(\bm z).
\end{equation}
Similarly, $\sigma$ is a $(\bm\gamma,\Pi)$-Fleming-Viot process if and only if $(N,\rho)$ is solution to the martingale problem for
\begin{equation}\label{eq:operator_for_N_rho}
\begin{aligned}
&\mathcal{L}^{\bm\gamma,\Pi} F(N,\rho)\\
&= 
\big( \gamma_b - \gamma_d \sigma(\mathbb{K})\big) \partial_N F(N,\rho)\\[10pt]
&\qquad \displaystyle+ \int_\mathbb{K}\int_\mathcal{Z} \left[ F\Big( (1-z_d)N + z_b, \left(1 - \overline{\bm z}\right)\rho + \overline{\bm z}\delta_\kappa\Big) - F(N,\rho)\right] \d\Pi(\bm z)\d\rho(\kappa),
\end{aligned}
\end{equation}
where we denote by
\[
\overline{\bm z} := \overline{\bm z}(N,\bm z) := \dfrac{z_b}{(1-z_d)N + z_b}
\]
the \emph{effective impact} of an event $\bm z$ at population size $N$. 
We keep the dependence on $N$ implicit to improve the readability.
In both cases, we abuse slightly the notation for $\mathcal{L}^{\bm\gamma,\Pi}$ by applying it to functions $f\in \mathcal{C}_b^1(\mathbb{R})$ and $F\in C_b^1((0,+\infty)\times \mathcal{M}_1(\mathbb{K}))$ respectively.
Furthermore, the equivalence holds only because the population does not vanish, which we will prove rigorously later on, see \Cref{sec:populationsize}.
%

\begin{figure}[!htb]
	\centering
	\includegraphics[width=0.8\linewidth]{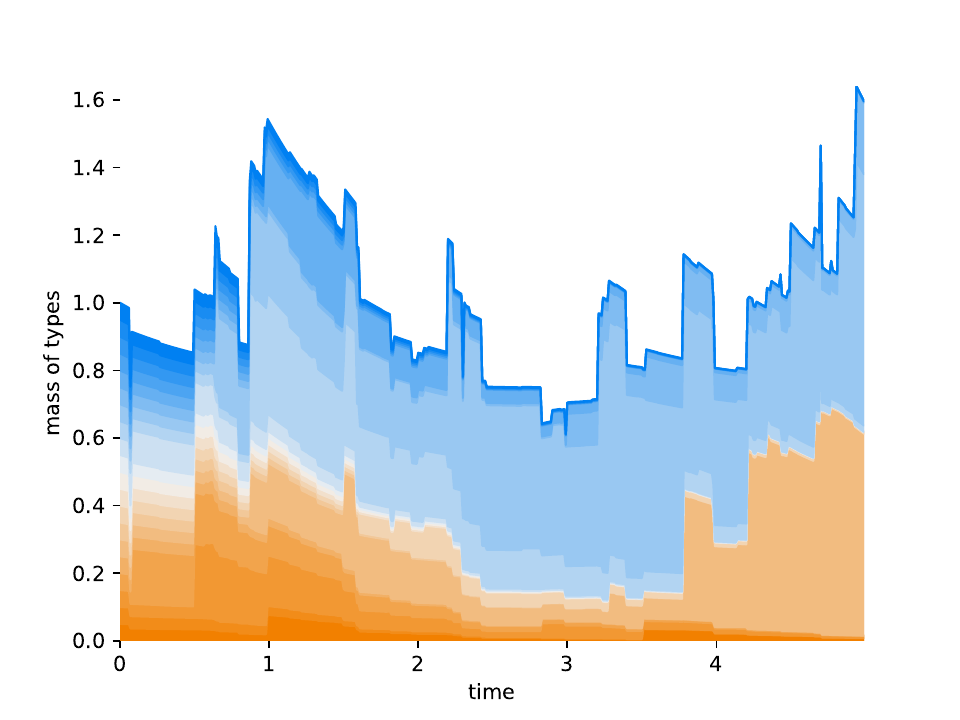}
	\caption{
		Simulation of a $(\gamma, \Pi)$ process: we start with $20$ different types of equal mass $0.05$, each represented by a different colour.
		Here, $\bm{\gamma} \approx (0.74, 1)$ and the marginals of $\Pi$ are independent such that $\Pi_{d}$ has density $\mathds{1}_{[0.0001, 0.3]} (x) \frac{3}{x}$, and $\Pi_{b}$ has density $ \propto \mathds{1}_{[0.0001, 0.4]} (x)\frac{1}{x}$. 		
	}
	\label{fig:piprocess}
\end{figure}

\subsection{Main results} \label{sec:mainresults}

\paragraph*{Properties of the model}
We show in \Cref{sec:well-posedness} that the model can be understood with the framework from \cite{DK99}.
In particular, the following result follows solely from studying the population process.
\begin{restatable}{theorem}{WellPosedness} \label{theo:well-posedness}
	The martingale problem for the $(\bm\gamma,\Pi)$-Fleming-Viot process is well-posed for any characteristic $(\bm{\gamma}, \Pi) \in \mathfrak{C}$. 
	Furthermore, the unique solution is a Hunt process, \emph{i.e} quasi-left continuous and strongly Markovian.
\end{restatable}

In the context of $\Lambda$-Fleming-Viot processes, reproduction events are parametrised by an \emph{impact} $u\in (0,1]$, representing the proportion of the population to be replaced.
Events of small impact $u\approx 0$ are usually modelled to occur at an accelerated rate of order $u^{-2}$.
Indeed, small reproduction events are expected to be much more frequent in large populations than large scale, catastrophic events with impact $u > 0$.
In the limit $u\downarrow 0$, the dynamics can be approximated by the Wright-Fisher diffusion which therefore can be thought of as ``$0$-impact Fleming-Viot".

As soon as events are balanced only in mean \eqref{eq:balance_condition}, events $\bm z \approx \bm 0$ can only happen at rate proportional to $\vert \bm z\vert^{-1}$ (see the discussion in \Cref{subsec:discussion}).
This is illustrated by the subordinator condition \eqref{eq:subordinator_condition} which ensures the existence of the $(\bm\gamma,\Pi)$-Fleming-Viot process.
Mirroring the approach from the $\Lambda$-Fleming-Viot process, it is natural to ask whether an additional process is needed to describe the ``$\bm 0$-impact" behaviour.
The next theorem gives a negative answer: the class of $(\bm\gamma,\Pi)$-models is closed under a natural notion of the convergence of characteristics.

For a Polish space $E$, denote by $\mathbb{D}_{[0,+\infty)}(E)$ the space of $E$-valued càdlàg processes with the usual Skorokhod topology.
On $\mathcal{M}_F(\mathbb{K})$, we consider the Polish topology of weak convergence induced by the topology on $\mathbb{K}$.

\begin{defin}\label{defin:convergence_characteristics}
For a sequence $\big((\bm\gamma^{(n)}, \Pi^{(n)})\big)_n\subset \mathfrak{C}$ of characteristics and a characteristic $(\bm\gamma,\Pi)\in\mathfrak{C}$, we write $(\bm\gamma^{(n)},\Pi^{(n)})\underset{n}{\longrightarrow} (\bm\gamma,\Pi)$ if and only if 
\begin{equation}\label{eq:Levy_condition}
\lim_n \bm\gamma^{(n)}\cdot\nabla g(\bm 0) + \int_\mathcal{Z} g(\bm z) \d\Pi^{(n)}(\bm z) = \bm\gamma\cdot\nabla g(\bm 0) + \int_\mathcal{Z} g(\bm z)\d\Pi(\bm z)
\end{equation}
for all $g\in C_b(\mathbb{R}^2)$ differentiable in $\bm 0$ such that $g(\bm 0) = 0$.
\end{defin}

\begin{restatable}{theo}{TheoremClosedness} \label{theo:closedness}
	The set of $(\bm{\gamma}, \Pi)$-Fleming-Viot processes is closed under limits of characteristics in $\mathfrak{C}$ in the following sense: if $\sigma_0^{(n)}$ converges weakly to some $\sigma_0\in\mathcal{M}_F(\mathbb{K})\setminus\{0\}$	and if $(\bm\gamma^{(n)},\Pi^{(n)})\rightarrow (\bm\gamma,\Pi)\in\mathfrak{C}$, then the sequence of $(\bm\gamma^{(n)},\Pi^{(n)})$-Fleming-Viot processes started in $\sigma_0^{(n)}$ converges weakly in $\mathbb{D}_{[0,+\infty)}\big(\mathcal{M}_F(\mathbb{K})\big)$ to the $(\bm\gamma,\Pi)$-Fleming-Viot process started in $\sigma_0$.
\end{restatable}

Finally, it is of interest to understand the relation between the $(\bm\gamma,\Pi)$- and the classical $\Lambda$-Fleming-Viot model. 
The next result illustrates that any $\Lambda$-Fleming-Viot process can be obtained by letting the jump measure concentrate on the diagonal $z_d = z_b$.

\begin{theo} \label{theo:convergencelambdafvprocess}
	Let $\big((\bm\gamma^{(n)}, \Pi^{(n)})\big)_n\subset \mathfrak{C}$ be a sequence of characteristics and $\Lambda$ be a finite measure on $[0,1]$. Assume that
	\[
	\lim_n \int_\mathcal{Z} g(\bm z) \d\Pi^{(n)}(\bm z) = \int_0^1 g(u,u)\;\dfrac{\mathrm{d}\Lambda(u)}{u^2}
	\]
	for all $g \in C_b (\mathbb{R}^2)$ vanishing in a neighbourhood of the origin and that
	\[
	\lim_{\epsilon\downarrow 0}\lim_n \int_{\mathcal{Z}\cap[0,\epsilon]^2} \left(\dfrac{z_b}{1 - z_d + z_b}\right)^2 \d\Pi^{(n)}(\bm z) = \Lambda(\{0\}).
	\]
	Let $\big((N_0^{(n)}, \rho_0^{(n)})\big)_n \subset (0, + \infty) \times \mathcal{M}_1(\mathbb{K})$ be a sequence with weak limit $(1, \rho_0) \in (0, + \infty) \times \mathcal{M}_1 (\mathbb{K})$, and denote by $(N^{(n)}, \rho^{(n)})$ the $(\bm\gamma^{(n)}, \Pi^{(n)})$-Fleming-Viot process started from $(N_0^{(n)}, \rho_0^{(n)})$.
	Then, $(N^{(n)},\rho^{(n)})$ converges weakly to $(1,\rho)$, where $\rho$ is a $\Lambda$-Fleming-Viot process started from $\rho_0$.
\end{theo}

\begin{figure}[!htb]
	\centering
	\includegraphics[width=0.9\linewidth]{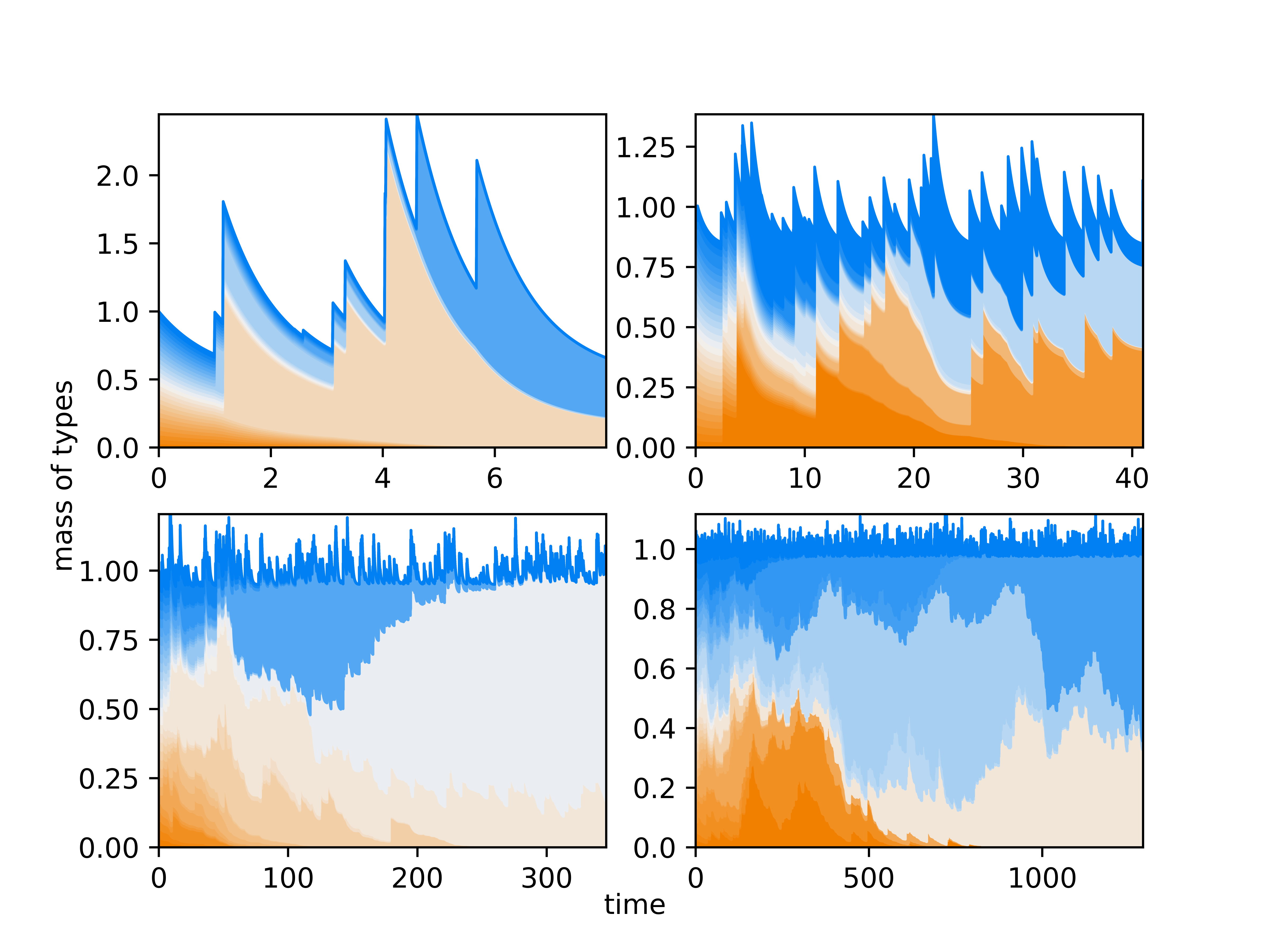}
	\caption{
		Illustration of \Cref{theo:convergencelambdafvprocess} in the case of a finite number of types. 
		The limiting process is a multi-type Wright-Fisher diffusion. 
		On the left-hand side of each figure, we start with $20$ different types of equal proportions, each represented by a different colour. 
		From the top left-hand figure to the bottom right-hand figure, we have used 
		\begin{equation*}
			\Pi_i := \delta_0\otimes \mathrm{Unif}([0,k_i])
		\end{equation*}
		where $k_1 = 1, k_2 = 1/3, k_3 = 1/10, k_4 = 1/20$. 
		The drift $\bm \gamma_i$ is given by $(\gamma_i)_d = 1$ and 
		\[
			(\gamma_i)_b = 1 - \int_\mathcal{Z} {z}_b \d{\Pi(\bm{z})} = 1 - \dfrac{k_i}{2}.
		\]
	}
	\label{fig:convergencetoWF}
	\vspace{-0.7cm}
\end{figure}

\paragraph*{Genealogy and duality}
As an infinite-population limit does not describe individuals any more, its genealogy can only be given a biological interpretation by considering the model as a limit of individual-based models.
In \Cref{subsec:genealogy}, we prove the convergence of an individual-based lookdown process to our model based on \cite[Theorem 3.2]{DK99}. 
This allows us to view the coalescent originating from tracing the lowest levels in the lookdown construction as providing the genealogy of the model.

The same coalescent can usually be used as a \emph{dual process} to the type distribution process.
Indeed, the lookdown construction from \Cref{ssec:contruction_random_environment} immediately implies a duality relation; see \Cref{subsec:duality}.
However, due to the dependence on the population size, the type distribution is not Markovian on its own. 
Using the lookdown construction, we will turn this non-Markovian duality into a pathwise duality in the sense of \cite{JK14}. 
Interestingly, this will provide us with an alternative proof of well-posedness of the $(\bm\gamma,\Pi)$-Fleming-Viot martingale problem that does not require the uniqueness of the associated lookdown process.
Furthermore, our method will provide us with a semi-explicit description of the coalescent, conditional on the driving noise.
We do not expect to obtain a more explicit description as the coalescent not only depends on the time reversal of the population size, but also on the offspring sizes not available from the population size process; see \Cref{subsec:duality} for details.

For this reason, we analyse two fundamental properties of the coalescent, \emph{fixation} and \emph{coming down from infinity}, by studying the genealogy through the lookdown process forwards-in-time.
We say that fixation occurs at time $t$ if the corresponding type distribution is a Dirac measure on a single type.
Through the lookdown construction, this is equivalent to proving that at time $t$ for every $n\in\mathbb{N}$, the $n$-th individual is a descendant of the lowest level individual.
We say that fixation occurs if the above random time is a.s.~finite.
Weakening the above, we say that quasi-fixation occurs if for every $n\in\mathbb{N}$, there exists an a.s.~finite time until the lowest $n$ levels are descendants from the lowest level.

\begin{restatable}[Quasi-fixation]{theo}{PropositionFixation} \label{pro:fixation}
	Consider the lookdown construction of the $(\bm{\gamma}, \Pi)$-Fleming-Viot process from \Cref{sec:well-posedness}. Assume that the population process $N_t$ is ergodic (cf. \Cref{lem:ergodicity}) and that the marginal $\Pi_b$ of $\Pi$ governing the births does not only charge zero.
	Then, for any $n \geq 2$, after an a.s.~finite time, the individuals associated to the lowest $n$ levels will be descendants of the individual associated to the lowest label.
\end{restatable}

Coming down from infinity describes the property that the coalescent started from infinitely many lineages contains only finitely many lineages after a finite amount of time.
Fixation cannot occur if the associated coalescent does not come down from infinity.
To complement the above, we show that the coalescent \emph{contains dust} at all times, i.e.~that the coalescent started from infinitely many lineages contains infinitely many singletons at all times and in particular does not come down from infinity.
For $\Lambda$-coalescents which are not star-shaped, i.e.~with $\Lambda(\{1\}) = 0$, this has been shown in the acclaimed paper \cite{Pit99} to be equivalent to
\begin{equation} \label{eq:containsdust}
	\int_0^1 u \; \dfrac{\Lambda (\d u)}{u^2} < \infty.
\end{equation}
Heuristically, this always holds for the $(\bm\gamma, \Pi)$-Fleming-Viot process, as we assume the subordinator condition \eqref{eq:subordinator_condition} for the measure $\Pi$.

\begin{restatable}[The coalescent contains dust]{theo}{PropositionNotComingDown} \label{pro:notcomingdown}
	Consider the lookdown construction of the $(\bm{\gamma}, \Pi)$-Fleming-Viot process from \Cref{sec:well-posedness}. 
	Then, for any time $t > 0$, the probability for the lowest lineage to not have had offspring before time $t$ is positive.
	
	If $\Pi_d$ does not charge $1$, this guarantees that, at every time $t$, there are a.s.~infinitely many ancestral lineages that do not coalesce in the time interval $[0,t]$.
\end{restatable}

\subsection{Discussion} \label{subsec:discussion}

The $(\bm\gamma,\Pi)$-Fleming-Viot model is a natural extension of the $\Lambda$-Fleming-Viot paradigm.
To make this connection apparent, let us consider the following intermediate example between the $\Lambda$- and the full $(\bm\gamma,\Pi)$-Fleming-Viot framework.

For small $u\approx 0$, let $\nu^u$ be a distribution on $[0,+\infty)$ with mean $u$.
In the following, $\nu^u[f(z_b)]$ will denote the expectation of $f$ w.r.t.~$z_b\sim \nu^u$.
The population size process of the $(\bm 0, \delta_u\otimes \nu^u)$-Fleming-Viot model, then, has the generator
\begin{align*}
	\mathcal{L}^{u} f(N) &= \nu^u\Big[ f\big((1-u)N + {z}_b^u\big) - f(N)\Big] \\
	&\approx \nu^u\Big[ f(N + {z}_b^u) - uN f'(N+{z}_b^u) - f(N) \Big]\\
	&\approx -uNf'(N) + \nu^u\Big[ f(N+{z}_b^u) - f(N)\Big],
\end{align*}
where we make use of the fact that $\mathbb{E}[f'(N+z_b)] \approx f'(N)$ as $u\downarrow 0$.
Whereas the first term is of order $u$, this is not necessarily true for the second.
Instead, we will make the natural assumption that it behaves in the limit like a Lévy process.
This is the same as asking that there are $\gamma_b, \sigma_b\geq 0$ and a $\sigma$-finite measure $\Pi_b$ on $(0,+\infty)$ satisfying $\int_0^{+\infty} 1\wedge{z}_b^2 \d{\Pi_b({z}_b)} < +\infty$ such that
\begin{equation}\label{eq:first_encounter_of_Levy_condition}
	\lim_{u\downarrow 0}\dfrac{1}{u}\nu^u[g({z}_b^u)] = \gamma_b g'(0) + \sigma_b g''(0) + \int_0^{+\infty} g({z}_b) \d{\Pi_b({z}_b)}
\end{equation}
for all $g:\mathbb{R}\rightarrow\mathbb{R}$ smooth enough and satisfying $g(0) = 0$. 
In our case, however, the jumps are all positive so that the stronger conditions $\sigma_b = 0$ and $\int_0^{+\infty} 1\wedge z_b \d\Pi_b(z_b) < +\infty$ hold.
This is a particular case of \Cref{eq:Levy_condition}.
Consequently, if we consider a Lévy subordinator $(L_t)_{t\geq 0}$ with characteristics $(\gamma_b,\Pi_b)$, the measures $\nu^u := Law(L_u)$ satisfy \Cref{eq:first_encounter_of_Levy_condition} with $\sigma_b = 0$.
Putting everything together, we obtain as a special case of \Cref{theo:closedness} that
\[
\mathcal{L}^u f(N) \approx u\cdot\left((\gamma_b - N)f'(N) + \int_0^{+\infty} f(N + z_b) - f(N) \d\Pi_b(z_b)\right).
\]

Adding a non-constant death proportion does not alter the above substantially.
Indeed, the only part that becomes more subtle is why we should not be able to see a diffusion in the limit as $z_b - z_dN$ can both be positive and negative.
This can be seen heuristically by the following Taylor expansion (that assumes both $z_d$ and $z_b$ to be of order $u$):
\begin{align*}
&\int_\mathcal{Z} f(N - z_dN + z_b) - f(N)\d\Pi^u(\bm z) \\
&\approx \int_\mathcal{Z} (z_b - z_dN)f'(N - z_dN + z_b) + (z_b - z_dN)^2f''(N - z_dN + z_b)\d{\Pi^u(\bm z)}\\
&\approx \left(\int_\mathcal{Z} (z_b - z_dN)\d\Pi^u(\bm z)\right)\cdot f'(N) + \left(\int_\mathcal{Z} (z_b - z_dN)^2\d\Pi^u(\bm z)\right)\cdot f''(N).
\end{align*}
If $\Pi^u$ is a probability measure, then the first integral is of order $u$ whereas the second is of order $u^2$.
In particular, whenever the first integral is non-zero, there is no way to speed up the dynamics to see a diffusion.
For a more complete treatment, see also \cite[Chapter VII, Section 3]{JS03}.
In the $\Lambda$-Fleming-Viot paradigm, this does not pose a problem as $N \equiv 1$ constantly so that for $z_b \equiv u \equiv z_d$ the first order term vanishes.
In the setting of a fluctuating population size, however, the first term can only vanish for at most one value of $N$. 
More precisely, under the balance condition \eqref{eq:balance_condition}, it vanishes exactly for $N=1$.
One should see a diffusion appear only in the setting where $N$ is forced to stay infinitesimally close to 1, i.e.~in the setting of \Cref{theo:convergencelambdafvprocess}, leading back to the $\Lambda$-Fleming-Viot model.
The particular choice of the balance condition \eqref{eq:balance_condition} is not relevant.
The fraction
\begin{equation}\label{eq:carrying_capacity}
K := \dfrac{\gamma_b + \int_\mathcal{Z} z_b \d\Pi(\bm z)}{\gamma_d + \int_\mathcal{Z} z_d\d\Pi(\bm z)}
\end{equation}
captures the \emph{carrying capacity} of the system to which the population experiences an exponential drift.
If we consider a $(\bm\gamma,\Pi)$-Fleming-Viot process $(N,\rho)$ such that it satisfies \Cref{eq:carrying_capacity} instead of the balance condition \eqref{eq:balance_condition}, then we may define
\[
\widehat{N}_t := \frac{N_t}{K},\qquad \widehat{\gamma} := \left(\gamma_d, \frac{\gamma_b}{K}\right)\quad\text{ and }\quad \int_\mathcal{Z} g(\bm z)\d{\widehat{\Pi}}(\bm z) := \int_\mathcal{Z} g\left(z_d, \frac{z_b}{K}\right)\d\Pi(\bm z)
\]
to obtain a $(\widehat{\gamma},\widehat{\Pi})$-Fleming-Viot process $(\widehat{N}, \rho)$ that has carrying capacity $1$.

\section{The population size process} \label{sec:populationsize}

Unsurprisingly, the population size process is crucial to the study of $(\bm\gamma, \Pi)$-Fleming-Viot processes. 
Hereafter, we will collect basic properties that we will use several times in the remainder of the paper. 

Let $(\bm\gamma, \Pi)\in\mathfrak{C}$ be a characteristic.
We will consider the operator \eqref{eq:operator_for_N}
\[
\mathcal{L}^{\bm\gamma,\Pi}f(N) = (\gamma_b - \gamma_dN)f'(N) + \int_\mathcal{Z} \Big( f\big( (1-{z}_d)N + {z}_b\big) - f(N)\Big) \d{\Pi(\bm{z})},
\]
defined for any $f\in \mathcal{C}_c^2( \mathbb{R} )$.

\begin{theo}[Well-Posedness]\label{theo:well-posedness_total_population}
	The martingale problem for $\mathcal{L}^{\bm\gamma,\Pi}$ is well-posed.
	Furthermore, the unique solution for a given initial condition is a conservative strong Markov process that is concentrated on $(0,+\infty)$ whenever the initial condition is in $(0, +\infty)$. 
	More precisely, if 
	\[\tau_\epsilon := \inf\{t\geq 0\;:\; N_t \not\in (\epsilon, 1/\epsilon)\text{ or  } N_{t-}\not\in (\epsilon, 1/\epsilon)\},
	\] 
	then
	\begin{equation}\label{eq:N_doesnt_hit_zero}
	\lim_{\epsilon\downarrow 0} \mathbb{P}\Big(\tau_\epsilon \leq T\Big) = 0
	\end{equation}
	for all finite time horizons $T\geq 0$.
	
	If $\bm L$ is a Lévy subordinator of characteristic $(\bm \gamma, \bm 0, \Pi)$, a solution can be constructed as the strong solution to
	\[
	\mathrm{d}N_t = \binom{-N_{t-}}{1}\cdot\mathrm{d}{\bm L_t}.
	\]
	Equivalently, one obtains the process by taking a Poisson point process $\xi$ on $[0,+\infty)\times \mathcal{Z}$ with mean intensity $\mathrm{d}t\otimes \Pi(\mathrm{d}\bm{z})$ and considering the unique strong solution to
	\[
	\mathrm{d}N_t = (\gamma_b - \gamma_d N_t)\mathrm{d}t + \int_\mathcal{Z} {z}_b - {z}_d N_{t-}\;\xi(\mathrm{d}t, \mathrm{d}\bm{z}).
	\]
	
	Furthermore, the process is a Feller process if and only if the first marginal $\Pi_d$ has no atom on $1$.
\end{theo}
\begin{proof}
Since $\mathcal{L}^{\bm\gamma,\Pi}$ maps $C_c^2(\mathbb{R})$ into $C_b(\mathbb{R})$, \cite[Theorem 2.3] {EquivSDEMart} ensures that any solution to the martingale problem on the one-point-compactification of $(0,+\infty)$ is a weak solution to the SDE
\[
\mathrm{d}N_t = \binom{-N_{t-}}{1}\cdot\mathrm{d}\bm L_t
\]
and vice-versa.
As the coefficient is globally Lipschitz, there exists a pathwise unique, non-exploding, strong solution to this SDE. 
Weak uniqueness follows and we may conclude the well-posedness of the martingale problem. 

Next, we prove \Cref{eq:N_doesnt_hit_zero}.
First, assume that $\Pi_d$ does not charge 1.
Enumerating the jumps of the driving noise $\bm L$, we may bound $N_t$ a.s.~by
\[
e^{-\gamma_d t} \left(\prod_{t_i \leq t} \left(1 - z_d^i\right)\right) N_0 \leq N_t \leq N_0 + \gamma_b t + \binom{0}{1}\cdot \bm L_t.
\]
Since the lower bound is non increasing and the upper bound is non decreasing, it suffices to show that the lower bound is strictly positive and the upper bound is finite.
For the upper bound, this is immediate.
For the lower bound, we note that the product is nonzero if and only if $\sum_{t_i\leq t} {z}_d^i$ is finite, which follows from
\[
\mathbb{E}\Bigg[ \sum_{t_i\leq t} {z}_d^i\Bigg] = t\int_{\mathcal{Z}} {z}_d \d\Pi(\bm{z}) < +\infty.
\]
Now, if $\Pi_d$ charges $1$, write $I$ for the set of jump indices satisfying $z_d^i = 1$.
If $i\in I$, then $N_{t_i}$ has the law $\nu_b = \Pi(\{1\}\times \cdot)/\Pi(\{1\}\times (0,+\infty))$.
In particular, we may bound
\[
N_t \geq e^{-\gamma_d T}\left(\prod_{\substack{t_i\leq T \\ i\not\in I}} (1-z_d^i)\right)\cdot \min(N_0, N_b),
\]
where $N_b\sim \nu_b$ is independent of both $N_0$ and the values of $z_d^i$, $i\not\in I$.
As before, we conclude that this random variable is a.s.~finite.

Finally, we apply \cite[Corollary 1.2]{K18a} which states that the Feller property is equivalent to
\[
\lim_{\vert N\vert \to +\infty} \Pi\left\{ \bm{z}\in\mathbb{R}^2\;:\; {z}_b - N{z}_d \in -B(N,r)\right\}  = 0
\]
for all $r > 0$, where $B(N,r)$ denotes the ball of radius $r$ around $N$. 
This is exactly the case if the first marginal of $\Pi$ does not charge the point $1$.
\end{proof}

The next result will be of practical interest when proving different tightness results.

\begin{corol}[Compact Containment]\label{cor:compact_containment_N}
Let $(\bm\gamma^{(\alpha)}, \Pi^{(\alpha)})_\alpha\subset \mathfrak{C}$ be a family of characteristics satisfying
\[
\sup_\alpha \left( \gamma^{(\alpha)}_b + \int_\mathcal{Z} {z}_b\d \Pi^{(\alpha)}(\bm{z})\right) < +\infty.
\]
Furthermore, suppose that the family $(N^{(\alpha)}_0)_\alpha$ of $(0,+\infty)$-valued random variables is tight. 
Let $(N_t^{(\alpha)})_{t\geq 0}$ be the population size process with characteristic $(\bm\gamma^{(\alpha)},\Pi^{(\alpha)})$ started in $N_0^{(\alpha)}$. 
Then, for any finite time horizon $T \geq 0$ and any threshold $\epsilon > 0$, there exists a compact $\Gamma\subseteq [0,+\infty)$ such that
\[
\inf_\alpha \mathbb{P}\Big( N^{(\alpha)}_t \in\Gamma\text{ for all }t\in [0,T]\Big) \geq 1 - \epsilon.
\]
\end{corol}
\begin{proof}
Let $\bm L^{(\alpha)}$ denote the Lévy subordinator corresponding to $N^{(\alpha)}$. 
As the strong solution to
\[
\mathrm{d}N_t^{(\alpha)} = \binom{-N_{t-}}{1}\cdot\mathrm{d}{\bm L^{(\alpha)}},
\]
the process $N^{(\alpha)}$ is almost surely bounded from above by the Lévy subordinator $L_t^{(\alpha, 2)} := \binom{0}{1}\cdot \bm L^{\alpha}$ of characteristic $\left(\gamma_b^{(\alpha)}, \Pi_b^{(\alpha)}\right)$.
Since subordinators are monotonically increasing, we may bound $L_t^{(\alpha, 2)}$ from above by $L_T^{(\alpha, 2)}$ for any $t\in[0,T]$. 
Now,
\[
\mathbb{E}\left[L_T^{(\alpha, 2)}\right] = T\left(\gamma_b^{(\alpha)} + \int_\mathcal{Z} {z}_b\d\Pi^{(\alpha)}(\bm{z})\right) \leq C_T
\]
for a constant $C_T$ depending only on $T$. Setting $\Gamma := \left[0, \frac{C_T}{\epsilon}\right]$, we obtain
\begin{align*}
\mathbb{P}\left( \sup_{t\in[0,T]} N_t^{(\alpha)} \in \Gamma\right) &\geq \mathbb{P}\Big( L_T^{(\alpha, 2)}\in \Gamma\Big) \geq 1 - \mathbb{P}\left( L_T^{(\alpha, 2)} \geq \dfrac{C_T}{\epsilon}\right) = 1 - \epsilon. \qedhere
\end{align*}
\end{proof}

We end this section with the long time behaviour of the population size process. 
This will be of interest when investigating whether there is fixation of a given type.
The total variation distance of two measures $\mu$ and $\nu$ on a common measurable space is defined as
\[
\Vert \mu - \nu\Vert_{TV} := \sup_A \vert \mu(A) - \nu(A)\vert,
\]
where the supremum ranges over all measurable sets.

\begin{prop}[Ergodicity] \label{lem:ergodicity}
	Suppose that $\gamma_d > 0$ and $\Pi\neq 0$.
	Then, there exists a unique invariant measure $\mu^*$ for the $(\gamma,\Pi)$-Fleming-Viot process. 
	Furthermore, this measure is concentrated on $(0,+\infty)$ and has a finite first moment in the sense that
	\[
	\int_0^{+\infty} n \d{\mu^*(n)} < +\infty.
	\]
	
	Whenever $N_0\sim \mu_0$ satisfies $\mathbb{E}[N_0] < +\infty$, the population size process $N$ of a $(\bm{\gamma}, \Pi)$-Fleming-Viot process started in $N_0$ is exponentially ergodic: if $\mu_t$ denotes the distribution of $N_t$, one has
		\[
		\Vert \mu_t - \mu^*\Vert_{TV} \leq C_{\mu_0}e^{-Ct}
		\]
		for two constants $C,C_{\mu_0} > 0$, where $C$ is independent of $\mu_0$.
\end{prop}
\begin{proof}
The existence of an invariant measure with finite first moment and the exponential ergodicity is an application of \cite{Kul09}, see Proposition 0.1 and Theorem 1.2 therein. From \Cref{eq:N_doesnt_hit_zero} and ergodicity, it follows that $\mu^*$ is concentrated on $(0,+\infty)$.
\end{proof}

It is to be expected that this results extends to more general conditions, where $\gamma_d = 0$ as long as $\Pi_d$ is non-zero.
Since this is not the main focus of this paper, it is left for future work.

\section{Closedness in the space of characteristics} \label{sec:closedness}
Recall that $(\sigma_t)_{t \geq 0}$ is a $(\bm{\gamma}, \Pi)$-Fleming-Viot process if and only if it is a $\mathcal{M}_F(\mathbb{K})$-valued solution to the martingale for $\mathcal{L}^{\bm\gamma,\Pi}$ defined in \Cref{eq:operator_for_sigma}.
We will make use of the well-posedness result from \Cref{sec:well-posedness} which does not depend on this section.

This section is concerned with the proof of \Cref{theo:closedness}.
For the rest of this section, we consider a weakly convergent sequence $\sigma_0^{(n)}\to \sigma_0 \in \mathcal{M}_F(\mathbb{K})\setminus\{0\}$ and some characteristics $(\bm\gamma^{(n)}, \Pi^{(n)})\to (\bm\gamma,\Pi)\in\mathfrak{C}$ in the sense of \Cref{defin:convergence_characteristics}.
Write $\sigma^{(n)}$ for the $(\bm\gamma^{(n)}, \Pi^{(n)})$-Fleming-Viot process started from $\sigma^{(n)}_0$.
The proof is divided into two lemmas.

\begin{lem}[Tightness]
The sequence $(\sigma^{(n)})_{n}$ is tight.
\end{lem}
\begin{proof}
\Cref{cor:compact_containment_N} guarantees the compact containement condition for the population size process, and thus for the process $\sigma$, as sets of measures that are uniformly bounded by a constant are compact. 
By \cite[Theorem 3.9.4]{EK86}, it is then sufficient to show the uniform boundedness of the generator in expectation over finite time intervals:
\begin{align*}
\left\vert \mathcal{L}^{\bm\gamma^{(n)}, \Pi^{(n)}} \Phi(\sigma)\right\vert &\leq \Big(\gamma_b^{(n)} + \gamma_d^{(n)}\sigma(\mathbb{K})\Big)\Vert D\Phi\Vert_\infty\\
&\qquad + \int_\mathcal{Z} \Vert D_\Phi\Vert_\infty\Big( {z}_d^{(n)}\sigma(\mathbb{K}) + {z}_b^{(n)}\Big)\d{\Pi^{(n)}(\bm{z})}\\
&\leq 2C\Vert D\Phi\Vert_\infty\cdot\big(1 + \sigma(\mathbb{K})\big).
\end{align*}
We conclude by noting that the corresponding population size process $N_t^{(n)} = \sigma_t^{(n)}(\mathbb{K})$ satisfies
\[
\mathbb{E}\left[N_t^{(n)}\right] \leq \mathbb{E}\left[N_0^{(n)}\right] + Ct,
\]
which is uniformly bounded on any compact time interval.
\end{proof}

\begin{lemma}
	Any limit point of the family $(\sigma^n)_{n\geq 1}$ is a $(\bm\gamma, \Pi)$-Fleming-Viot process.
\end{lemma}
\begin{proof}
	By \parencite[Theorem 4.8.10]{EK86} it suffices to show
	\begin{equation} \label{eq:identificationoflimit}
		\lim_{n \to \infty} \mathbb{E} \Bigg[ \Bigg( \Phi (\sigma_{t+\tau}^{(n)}) - \Phi (\sigma_t^{(n)}) - \int_t^{t+\tau} \mathcal{L}^{\bm{\gamma}, \Pi} \Phi (\sigma_s^{(n)}) \d s \Bigg) \prod_{i = 1}^k h_i (\sigma_{t_i}^{(n)}) \Bigg] = 0
	\end{equation}
	for all $k \geq 0, 0 \leq t_1 < t_2 < ... < t_k \leq t < t + \tau$ and $h_1,..., h_k \in C_b \mathcal{M}_F (\mathbb{K})$. 
	
	\Cref{eq:identificationoflimit} is satisfied in the prelimit in the sense that
	\begin{equation*} 
		\lim_{n \to \infty} \mathbb{E} \Bigg[ \Bigg( \Phi (\sigma_{t+\tau}^{(n)}) - \Phi (\sigma_t^{(n)}) - \int_t^{t+\tau} \mathcal{L}^{\bm{\gamma}^{(n)}, \Pi^{(n)}} \Phi (\sigma_s^{(n)}) \d s \Bigg) \prod_{i = 1}^k h_i (\sigma_{t_i}^{(n)}) \Bigg] = 0.
	\end{equation*}
	This reduces the statement necessary for proving the result to 
	\begin{equation*}
		\lim_{n \to \infty} \mathbb{E} \Bigg[ \int_t^{t + \tau} \Big\vert \mathcal{L}^{\bm{\gamma}^{(n)} , \Pi^{(n)}} \Phi (\sigma_s^{(n)}) - \mathcal{L}^{\bm{\gamma}, \Pi} \Phi (\sigma_s^{(n)}) \Big\vert \d s \prod_{i = 1}^k h_i (\sigma_{t_i}^{(n)}) \Bigg] = 0.
	\end{equation*}
	Using the tower property of conditional expectation, this statement is satisfied provided
	\begin{equation} \label{eq:identificationoflimitsufficientequation}
		\lim_{n \to \infty} \mathbb{E} \Bigg( \int_t^{t + \tau} \Big\vert \mathcal{L}^{\bm{\gamma}^{(n)}, \Pi^{(n)}} \Phi (\sigma_s^{(n)}) - \mathcal{L}^{\bm{\gamma}, \Pi} \Phi (\sigma_s^{(n)}) \Big\vert \d s \enspace \Bigg\vert \enspace \mathcal{F}_t \Bigg) =0.
	\end{equation}
	Let us introduce the family of functions
	\begin{equation*}
		f_\Phi^\sigma: \mathbb{R}^2 \rightarrow \mathbb{R},\quad ({z}_1, {z}_2) \mapsto \int_{\mathbb{K}} \Phi \big( (1 - {z}_2) \sigma + {z}_1 \delta_\kappa \big) - \Phi (\sigma) \frac{\d \sigma (\kappa)}{\sigma (\mathbb{K})},
	\end{equation*}
	which are continuously differentiable and satisfy $f_\Phi^\sigma(\bm{0}) = 0$.
	The following identity holds for such functions:
	\begin{equation*}
		\nabla f_\Phi^\sigma (\bm{0}) = D \Phi \left(\sigma, \frac{\sigma}{\sigma (\mathbb{K})}\right) \binom{1}{ - \sigma (\mathbb{K})}.
	\end{equation*}
	and thus
	\begin{equation*}
		(\gamma_b - \gamma_d \sigma (\mathbb{K})) D \Phi \Big( \sigma, \frac{\sigma}{\sigma (\mathbb{K})} \Big) = \bm\gamma \cdot \nabla f_\Phi (\bm{0}).
	\end{equation*}
	For $\sigma \in \mathcal{M}_F(\mathbb{K})$ the difference between the prelimiting and limiting generator is given by
	\begin{align*}
		&\Big\vert \mathcal{L}^{\bm{\gamma}^{(n)}, \Pi^{(n)}} \Phi (\sigma) - \mathcal{L}^{\bm{\gamma}, \Pi} \Phi (\sigma) \Big\vert\\
		&\leq \Bigg\vert \Big( \gamma_b^{(n)} - \gamma_d^{(n)} \sigma (\mathbb{K})\Big) D \Phi \Big(\sigma, \frac{\sigma}{\sigma(\mathbb{K})} \Big) - \Big( \gamma_b - \gamma_d \sigma (\mathbb{K})\Big) D \Phi \Big(\sigma, \frac{\sigma}{\sigma (\mathbb{K})} \Big) \\
		&\hspace{1cm} + \int_\mathcal{Z} \int_\mathbb{K} \left[ \Phi\Big( (1-{z}_d)\sigma + {z}_b\delta_\kappa\Big) - \Phi(\sigma)\right]  \dfrac{\mathrm{d}\sigma(\kappa)}{\sigma(\mathbb{K})} \d{\Pi^{(n)}(\bm{z})}   \\
		&\hspace{1cm} - \int_\mathcal{Z} \int_\mathbb{K} \left[ \Phi\Big( (1-{z}_d)\sigma + {z}_b\delta_\kappa\Big) - \Phi(\sigma)\right]  \dfrac{\mathrm{d}\sigma(\kappa)}{\sigma(\mathbb{K})} \d{\Pi(\bm{z})}   \Bigg\vert  \\
		&\leq \Bigg\vert \bm{\gamma}^{(n)} \cdot \nabla f_\Phi (\bm{0}) + \int_\mathcal{Z} f_\Phi (\bm{{z}}) \d{\Pi^{(n)}(\bm{z})}   - \bm{\gamma} \cdot \nabla f_\Phi (\bm{0}) - \int_\mathcal{Z} f_\Phi (\bm{{z}}) \d{\Pi(\bm{z})}  \Bigg\vert.
	\end{align*}
	Hence, by the convergence of the characteristics, we obtain the pointwise limit
	\begin{equation*}
		\lim_{n \to \infty} \Big\vert \mathcal{L}^{\bm{\gamma}^{(n)}, \Pi^{(n)}} \Phi (\sigma) - \mathcal{L}^{\bm{\gamma}, \Pi} \Phi (\sigma) \Big\vert = 0.
	\end{equation*}
	Writing $N := \sigma (\mathbb{K})$ for the population size, we have
	\begin{align*}
		\Big\vert \mathcal{L}^{\bm{\gamma}^{(n)}, \Pi^{(n)}} \Phi (\sigma) - \mathcal{L}^{\bm{\gamma}, \Pi} \Phi (\sigma) \Big\vert &\leq \Vert D \Phi \Vert_\infty \Big( \sup_n \vert \gamma_b^{(n)} - \gamma_b \vert + \sup_n \vert \gamma_d - \gamma_d^{(n)} \vert N \Big) \\
		&\hspace{1cm} + \sup_n 2 \Vert D \Phi \Vert_\infty \int_{\mathcal{Z}} ({z}_b + {z}_d) \d \Pi^{(n)}(\bm{{z}}),
	\end{align*}
	which has finite mean. 
	The above is an integrable majorant and we may apply Lebesgue's Dominated Convergence to conclude.
\end{proof}

\section{Relation to the $\Lambda$-Fleming-Viot process} \label{sec:relationlambdafvprocess}
This section is devoted to the proof of \Cref{theo:convergencelambdafvprocess}, which illustrates that $(\bm\gamma, \Pi)$-Fleming-Viot processes can approximate $\Lambda$-Fleming-Viot processes arbitrarily well. 
We prove the result in two steps.
First, we prove that the population size concentrates on 1 by viewing it as the solution to an SDE with large drift.
Then, we verify the convergence of the type distribution.
More precisely, making use of \cite[Lemma 5.3]{BES04}, it is enough to show the convergence of the type composition up to the stopping time $\tau_\delta := \inf\{ t\geq 0\;:\; N_t\leq \delta\text{ or }N_{t-}\leq \delta\}$.

\begin{lem} \label{lem:convergencepopsize}
	Let $\big((\bm\gamma^{(n)}, \Pi^{(n)})\big)_n \subset \mathfrak{C}$ be a sequence of characteristics satisfying the assumptions of \Cref{theo:convergencelambdafvprocess}. Then the associated sequence $(N^{(n)})_n$ of population size processes converges weakly to $1$ in $\mathbb{D}(\mathbb{R})$.
\end{lem}
\begin{lem} \label{lem:convergencetypes}
Let $\big((\bm\gamma^{(n)}, \Pi^{(n)})\big)_n \subset \mathfrak{C}$ be a sequence of characteristics satisfying the assumptions of \Cref{theo:convergencelambdafvprocess} for a finite measure $\Lambda$ on $[0,1]$, and suppose that $(N_0^{(n)}, \rho_0^{(n)})\to (1, \rho_0)$ weakly. 
Denote by $\rho^{(n)}$ the associated sequence of measure-valued type processes and write $\rho$ for the $\Lambda$-Fleming-Viot process started in $\rho_0$. 
For every $\delta > 0$, the sequences of processes $\rho^{(n)}$ stopped at $\tau_\delta$ converges weakly to $\rho$ stopped at $\tau_\delta$.
\end{lem}
As the total population size process converges to a deterministic function, the convergence of the components suffices for the joint convergence of $(N^{(n)}, \rho^{(n)})$ and \Cref{theo:convergencelambdafvprocess} follows.

\begin{proof}[Proof of \Cref{lem:convergencepopsize}]
	If $\Lambda$ satisfies the subordinator condition, in the sense that $$\int_0^1 u^{-1}\d\Lambda(u) < +\infty,$$ and $\sup_n \vert \bm\gamma\vert < +\infty$, then  necessarily $\gamma_d^{(n)} - \gamma_b^{(n)} \to 0$.
	In particular, the convergence follows from the closedness of the model, see \Cref{theo:closedness}.
	The remainder of the proof relies on \cite{K91} which considers stochastic processes that are forced onto a manifold by a large drift.
	
	To this end, suppose from now on that $C_n := \gamma_b^{(n)} + \int_\mathcal{Z} z_b\d\Pi^{(n)}(\bm z)\to+\infty$.
	From \Cref{theo:well-posedness_total_population}, we have that the population size process is a solution to
	\begin{align*}
		N_t^{(n)} &= N_0^{(n)} + C_n\int_0^t ( 1 - N_t^{(n)} )\d t + \int_0^t \binom{-N}{1}\cdot \d{\bm L^{(n)}_t},
	\end{align*}
	where $\bm L^{(n)}$ is a Lévy process with characteristic $\left( -\int_\mathcal{Z} \bm z\d\Pi^{(n)}(\bm z), \bm 0, \Pi^{(n)}\right)$.
	In this setting, the second term forces the process onto the singleton manifold $\{1\}$.
	Note that $\bm L^{(n)}$ is a martingale with quadratic variation
	\begin{align*}
		[\bm L^{(n)}]_t = t\int_\mathcal{Z} \bm z^2 \d\Pi^{(n)}(\bm z),
	\end{align*}
	which is uniformly bounded in $n$, and thus uniformly integrable.
	In particular, \cite[Condition (4.2)]{K91} is satisfied for $Y = \bm L^{(n)}$ by choosing $\delta = \infty$ and $M = \bm L^{(n)}$. This implies that \cite[Condition (C5.1)]{K91} holds, whereas \cite[Condition (C5.2)]{K91} follows automatically. We apply \cite[Theorem 6.2]{K91} to obtain the convergence $N^{(n)} \to 1$.
\end{proof}

\begin{proof}[Proof of Lemma \ref{lem:convergencetypes}]
	The generator of the Fleming-Viot process is given by
	\[
	A_1\Phi(\rho) := \dfrac{1}{2}\int_\mathbb{K}\int_\mathbb{K} D^2\Phi(\rho; \delta_\kappa,\delta_\chi) \;\left(\delta_\kappa - \rho\right)(\mathrm{d}{\chi}) \; \rho(\mathrm{d}\kappa),
	\]
	where $D^2$ denotes the second order Fréchet derivative. The generator is defined for all functions $\Phi\in\mathcal{C}(\mathcal{M}_1(\mathbb{K}))$ for which the second derivative exists for all $\kappa,\chi\in\mathbb{K}$ and is bounded continuous. The generator of the $\Lambda$-Fleming-Viot is $A \Phi (\rho) = A_1^\Lambda \Phi (\rho) + A_2^\Lambda \Phi (\rho)$, where 
	\begin{equation*}
		A_1^\Lambda := \Lambda(\{0\})A_1\quad\text{ and }\quad
		A_2^\Lambda \Phi(\rho) := \int_0^1 \int_{\mathbb{K}} \Big[ \Phi\big( (1- u)\rho + u\delta_\kappa\big) - \Phi(\rho)\Big] \d{\rho(\kappa)} \frac{\d \Lambda (u)}{u^2}.
	\end{equation*}
	Recall that the generator of the type distribution process $\rho^{(n)}$ is given by
	\begin{align*}
A^{(n)} \Phi(N,\rho) &= \int_\mathcal{Z}\int_\mathbb{K} \Phi\Big( (1-\overline{\bm z})\rho + \overline{\bm z}\delta_\kappa\Big) - \Phi(\rho)\d\rho(\kappa)\d\Pi^{(n)}(\bm z)\\
&= \int_{\mathcal{Z}\cap [0,\epsilon]^2} \int_\mathbb{K} \Phi\Big( (1-\overline{\bm z})\rho +  \overline{\bm z}\delta_\kappa\Big) - \Phi(\rho)\d\rho(\kappa)\d\Pi^{(n)}(\bm z)\\
&\qquad + \int_{\mathcal{Z}\setminus [0,\epsilon]^2} \int_\mathbb{K} \Phi\Big( (1-\overline{\bm z})\rho + \overline{\bm z}\delta_\kappa\Big) - \Phi(\rho) \d\rho(\kappa)\d\Pi^{(n)}(\bm z)\\
&=: A_1^{(n,\epsilon)} + A_2^{(n,\epsilon)}.
	\end{align*}
	with $\overline{\bm{{z}}} := \frac{{z}_b}{(1-{z}_d)N + {z}_b}$ for functions $\Phi\in C_b^1(\mathcal{M}_1(\mathbb{K}))$ depending only on $\rho$.
	Adapting \cite[Theorem 2.3]{BK93} to our situation, the statement follows from
	\begin{enumerate}
		\item 
		the sequence $(\rho^{(n)})_n$ is tight in $\mathbb{D}_{[0,T]}(\mathcal{M}_1(\mathbb{K}))$,
		\item
		and for all $\Phi: \mathcal{M}_1(\mathbb{K}) \rightarrow \mathbb{R} $ such that $\Phi$ is three times Fréchet differentiable with bounded derivatives,
		\begin{equation} \label{eq:boundedgenerator}
			\sup_{n \to \infty} \mathbb{E}\left[ \sup_{t\in [0,T]} \big\Vert A^{(n)} \Phi(N^{(n)}, \cdot) \big\Vert_\infty \right] \leq C_\Phi
		\end{equation}
		for $0 < C_\Phi < +\infty$ and for any $t\in [0,T]$
		\begin{equation} \label{eq:generatorconvergence}
			\lim_{n \to \infty} \mathbb{E}\left[ \big\Vert A^{(n)} \Phi(N^{(n)}, \cdot) - A\Phi \big\Vert_\infty \right]= 0.
		\end{equation}
	\end{enumerate}
	These conditions present a slight modification of \cite[Theorem 2.3]{BK93}, where we have used that the three times Fréchet differentiable functions are dense in the domain of $A$ and that the proof remains valid if \eqref{eq:boundedgenerator} and \eqref{eq:generatorconvergence} hold in expectation. Now, tightness of $(\rho^{(n)})_n$ follows from \eqref{eq:boundedgenerator} and the compactness of $\mathcal{M}_1(\mathbb{K})$, so that we will concentrate on proving \eqref{eq:boundedgenerator} and \eqref{eq:generatorconvergence}.

	Applying Taylor's expansion to the inner part, we obtain
	\begin{align*}
		&A^{(n, \epsilon)}_1\Phi(N, \rho) \\
		&= \int_{\mathcal{Z} \cap [0, \epsilon]^2} \int_\mathbb{K} \left[ \overline{\bm{{z}}}\cdot D\Phi(\rho; \delta_\kappa - \rho) + \overline{\bm{{z}}}^2\cdot \dfrac{1}{2}D^2 \Phi(\rho; \delta_\kappa-\rho, \delta_\kappa-\rho) + \mathcal{O}\left(\overline{\bm{{z}}}^3\right)\right] \d{\Pi^{(n)}(\bm{z})}.
	\end{align*}
	The first summand is equal to zero due to the identity
	\[
	\int_\mathbb{K} D\Phi(\rho; \delta_\kappa -\rho) \d{\rho(\kappa)} = \int_\mathbb{K}\int_\mathbb{K} \Big[ D\Phi(\rho;\delta_\kappa) - D\Phi(\rho;\delta_\chi)\Big] \d{\rho(\kappa)}\d{\rho(\chi)} = 0.
	\]
	The second summand equals
	\begin{align*}
		&\int_\mathbb{K} D^2\Phi(\rho; \delta_\kappa - \rho, \delta_\kappa - \rho) \d{\rho(\kappa)}\\
		&= \int_\mathbb{K} \Big[ D^2\Phi (\rho;\delta_\kappa, \delta_\kappa - \rho) - D^2\Phi(\rho; \rho, \delta_\kappa - \rho)\Big] \d{\rho(\kappa)}\\
		&= \int_\mathbb{K} \int_\mathbb{K} D^2\Phi(\rho; \delta_\kappa, \delta_\chi) \; (\delta_\kappa - \rho)(\mathrm{d}\chi)\d{\rho(\kappa)} \\
		&\hspace{1cm} - \int_\mathbb{K}\int_\mathbb{K}\int_\mathbb{K} \Big[ D^2\Phi(\rho; \delta_\eta, \delta_\kappa) - D^2\Phi(\rho; \delta_\eta, \delta_\theta)\Big] \d{\rho(\eta)}\d{\rho(\theta)}\d{\rho(\kappa)}\\
		&= \int_\mathbb{K} \int_\mathbb{K} D^2\Phi(\rho; \delta_\kappa, \delta_\chi) \; (\delta_\kappa - \rho)(\mathrm{d}\chi)\d{\rho(\kappa)} = A\Phi(\rho)
	\end{align*}
	and the third can be bounded by
	\[
	\dfrac{\epsilon}{(1-\epsilon)^3}\dfrac{1}{N^3}\cdot \sup_n\int_{\mathcal{Z}} z_b^2 \d\Pi^{(n)}(\bm z).
	\]
	Altogether,
	\[
	A^{(n, \epsilon)}_1 \Phi(N,\rho) = \left(\int_{\mathcal{Z} \cap [0, \epsilon]^2}  \overline{\bm{{z}}}^2  \d{\Pi^{(n)}(\bm{z})}\right) \cdot A_1\Phi(\rho) + \mathcal{O}\left(\dfrac{\epsilon}{N^3}\right)
	\]
	which is bounded up to the stopping time $\tau_\delta$.
	The boundedness of $A_2^{(n,\epsilon)}$ is immediate from the fact that $\Pi^{(n)}$ is finite outside a neighbourhood of $0$:
	\begin{equation*}
	\vert A_2^{(n, \epsilon)} \Phi (N, \rho) \vert \leq 2 \Vert \Phi \Vert_\infty \sup_n \int_{\mathbb{Z} \setminus [0, \epsilon]^2} \d \Pi^{(n)} (\bm{z}) < + \infty.
	\end{equation*}
	
To prove \eqref{eq:generatorconvergence} we first bound
\begin{equation} \label{eq:generatorconvergenceinequality}
\begin{aligned}
	&\lim_n \mathbb{E}\Big[ \big\Vert A^{(n)}\Phi(N^{(n)}, \cdot) - A\Phi \big\Vert_\infty\Big] \\
	&\leq \limsup_{\epsilon\downarrow 0}\lim_n \mathbb{E}\Big[ \big\Vert A^{(n,\epsilon)}_1\Phi -  A_1^\Lambda \Phi \big\Vert_\infty + \big\Vert A^{(n,\epsilon)}_2\Phi - A^{(\Lambda, \epsilon)}_{2,2}\Phi \big\Vert_\infty + \big\Vert A^{(\Lambda,\epsilon)}_{2,1} \Phi \big\Vert_\infty\Big],
\end{aligned}
\end{equation}
where
\begin{align*}
	A_{2, 1}^{(\Lambda,\epsilon)} \Phi(\rho) &:= \int_0^\epsilon \int_{\mathbb{K}} \Big[ \Phi\big( (1- u)\rho + u\delta_\kappa\big) - \Phi(\rho)\Big] \d{\rho(\kappa)} \frac{\d \Lambda (u)}{u^2},\\
	A_{2,2}^{(\Lambda,\epsilon)} \Phi(\rho) &:= \int_\epsilon^1 \int_{\mathbb{K}} \Big[ \Phi\big( (1- u)\rho + u\delta_\kappa\big) - \Phi(\rho)\Big] \d{\rho(\kappa)} \frac{\d \Lambda (u)}{u^2}.
\end{align*}
The first summand of inequality \eqref{eq:generatorconvergenceinequality} is bounded by
\begin{align*}
& \big\vert A_1^{(n,\epsilon)}\Phi (N, \rho) - A_1^\Lambda\Phi (\rho) \big\vert \\
&\leq \Bigg\vert \int_{\mathcal{Z} \cap [0,\epsilon]^2} \Bigg[ \left(\dfrac{{z}_b}{(1-{z}_d)N_t^{(n)} + {z}_b}\right)^2 - \left(\dfrac{z_b}{1 - z_d + z_b}\right)^2 \Bigg] \d\Pi^{(n)}(\bm z) \Bigg\vert \big\Vert A_1\Phi \Vert_\infty\\
& \quad + \left\vert \int_{\mathcal{Z}\cap [0,\epsilon]^2} \left(\dfrac{z_b}{1 - z_d + z_b}\right)^2\d\Pi^{(n)}(\bm z) - \Lambda(\{0\}) \right\vert  \big\Vert A_1\Phi \big\Vert_\infty + \mathcal{O}\left(\dfrac{\epsilon}{N^3}\right)
\end{align*}
Here, the second part vanishes due to the assumption on $\Lambda(\{ 0 \})$. The first and the third parts can be treated up to $\tau_\delta$ by using the convergence of the population size and by noting that
	\begin{equation*}
	\begin{aligned}
		&\Bigg\vert \int_{\mathcal{Z}\cap [0,\epsilon]^2} \left[ \left(\dfrac{{z}_b}{(1-{z}_d)N_t^{(n)} + {z}_b}\right)^2 - \left(\dfrac{{z}_b}{1 - {z}_d + {z}_b}\right)^2 \right]\d{\Pi^{(n)}(\bm{z})} \Bigg\vert\\
		&\leq \Bigg\vert \int_{\mathcal{Z}\cap [0,\epsilon^2]} \dfrac{2 z_b \vert 1 - N_t^{(n)}\vert}{(1-\epsilon)^2N_t^{(n)}}\cdot  \dfrac{z_b}{1-\epsilon} \left(\dfrac{(1 + 2 \epsilon)}{ (1- \epsilon)N_t^{(n)}} + 1\right)\d{\Pi^{(n)}(\bm{z})} \Bigg\vert\\
		&\leq \dfrac{2 \vert 1 - N_t^{(n)}\vert}{(1-\epsilon)^3N_t^{(n)}} \left(\dfrac{(1 + 2 \epsilon)}{ (1- \epsilon)N_t^{(n)}} + 1\right) \sup_n \int_{\mathcal{Z} \cap [0, \epsilon]^2} z_b^2 \d \Pi^{(n)} (\bm z).
	\end{aligned}
\end{equation*}
The convergence of the second summand of inequality \eqref{eq:generatorconvergenceinequality} can be achieved in a similar fashion, adding and subtracting a term without the population size and using Taylor's expansion. The third summand of \eqref{eq:generatorconvergenceinequality} can be expanded similarly to $A_1^{(n, \epsilon)}$ above and is bounded because of the boundedness of the derivatives of $\Phi$. Taking limits, the last term vanishes as well. 
\end{proof}

\section{Lookdown construction} \label{sec:well-posedness}
This section provides a proof for \Cref{theo:well-posedness}, i.e.~the existence and uniqueness of the $(\bm\gamma, \Pi)$-Fleming-Viot process.
The main part will be the construction of a lookdown representation, from which uniqueness will carry over to the original process.
This is a special case of \cite[Section 4]{DK99}, see also \cite[Section 2]{alphabeta} for a similar construction. 
We will reframe the result in the language of the Markov Mapping Theorem from \cite{Kur98}, see also \cite{NK11,EK19}.

A lookdown process allows one to represent a measure-valued population process as a countable system of particles, labelled by some \emph{levels} in $\mathbb{N}$. 
The name lookdown comes from the fact that children typically look down on their parents in the sense that individuals with higher levels inherit the type of individuals with lower level. 
However, even though the dynamics of the labelled population can depend on the labels, care has to be taken as to keep the particles exchangeable.
This allows the measure-valued process to be retrieved from the lookdown process as the limiting empirical average (or DeFinetti measure).

If the individual-based prelimiting model can be used to define the lookdown process, the latter also contains the genealogy and can be used to study the corresponding coalescent.
In this case, the lookdown even provides a forward-in-time option to study the lines of descent.

\subsection{Construction as a $\Lambda$-Fleming-Viot in a random environment}\label{ssec:contruction_random_environment}
The following lookdown representation was first introduced in a general form in \cite[Section 3.3]{DK99}. 
Here, we follow the presentation of \cite{alphabeta}, where a special case has been treated.
In the following, we will write $(N,\rho)$ to denote a $(\bm\gamma,\Pi)$-Fleming-Viot process. 
More generally, we use the symbol $N$ for a population size and $\rho$ for a type distribution on $\mathbb{K}$. 
In this section, we consider a system of labelled particles as a sequence $\mathsf{k}\in \mathbb{K}^\mathbb{N}$ in the sense that $i$ is the label of the individual $\mathsf{k}(i)$.
We will describe a particle process $(\mathsf{k}_t)_{t\geq 0}$ such that $\mathsf{k}_t$ is conditionally independent with DeFinetti measure $\rho_t$.
Since this sequence does not contain any information on the actual \emph{size} of the population, we will additionally consider the population size process $N$.

As we noted in \Cref{eq:operator_for_N_rho}, the type distribution does not see the continuous births and deaths and evolves only through reproduction events.
If we denote by $\overline{\bm z} := \overline{\bm z}(N, \bm z) := \frac{z_b}{(1-z_d)N + z_b}$ the \emph{effective impact} of the event $\bm z$ at population size $N$, the event $\bm z$ affects the type distribution via 
\[
\rho_t = \Big(1 - \overline{\bm z}(N_{t-}, \bm z)\Big)\rho_{t-} + \overline{\bm z}(N_{t-}, \bm z)\delta_\kappa,
\]
where $\kappa \sim \rho_{t-}$.
To achieve this on the level of the lookdown, we will assign every individual $i$ independently a uniform variable $\mathsf{u}(i)\sim \mathrm{Unif}([0,1])$.
Since a proportion $\overline{\bm z}$ is to be replaced, we mark an individual $i$ whenever $\mathsf{u}(i) \leq \overline{\bm z}$.
Since we suppose the sequence $\mathsf{k}$ to be exchangeable, we may choose the parent to be the marked individual with the lowest level.
Note that this means that the parent is always killed and replaced by its offspring, but this does not impact an infinite population.
From here, we would like to simply replace all marked individuals by offspring with the type of the parent.
However, this would destroy the exchangeability of the particle system. 
Instead, we will \emph{insert} the offspring at the marked levels and shift all other individuals upwards accordingly, see \Cref{fig:lookdown_event}.

\newarray\typesBefore
\readarray{typesBefore}{purple&purple&green&green&purple&green}
\dataheight=1
\definecolor{green}{rgb}{1,0,0}
\definecolor{purple}{rgb}{0,0,1}
\begin{figure}
\centering
	\begin{tikzpicture}
			\node at (0, 8.5*0.8) {Label};
		\node at (0, 7*0.8) {$\vdots$};
		\node at (3, 8.5*0.8) {Types before event};
		\node at (3, 7*0.8) {$\vdots$};
		\foreach \i in {1,...,6} {
			\node at (0,\i*0.8) {\i};
			\checktypesBefore(\i)
			\trimspace\cachedata
			\draw[draw=\cachedata] (1,\i*0.8) -- 	(5,\i*0.8);
		}
		
		\foreach \i in {2,4,5} {
			\node at (5, \i*0.8) {$\bigtimes$};
			
			\checktypesBefore(2)
			\trimspace\cachedata
			\draw[draw=\cachedata] (7, \i*0.8) -- 	(11,\i*0.8);
			\draw[->] (5.2, 2*0.8) -- (6.8, \i*0.8);
		}
		
		\foreach \i in {1, 3, 6} {
			\checktypesBefore(\i)
			\trimspace\cachedata
			\draw[draw=\cachedata] (7, \i*0.8) -- (11,\i*0.8);
		}
		\draw[->, dashed] (5.2,1*0.8) -- (6.8,1*0.8);
		\draw[->, dashed] (5.2, 3*0.8) -- (6.8,3*0.8);
		\draw[->, dashed] (5.2, 4*0.8) -- (6.8, 6*0.8);
		\draw[->, dashed] (5.2, 5*0.8) -- (6.8, 7*0.8);
		\draw[->, dashed] (5.2,6*0.8) -- (6.8,8*0.8);
		\checktypesBefore(5)
		\trimspace\cachedata
		\draw[draw=\cachedata] (7, 7*0.8) -- (11, 7*0.8);
		\node at (9, 8.5*0.8) {Types after event};
		\node at (9, 8*0.8) {$\vdots$};
	\end{tikzpicture}
	\caption{Illustration of a reproduction event: during the event, the particles with labels $2,4,5$ are chosen to participate. Offspring of the parental type are inserted. All other levels have to shift upwards accordingly. Types are marked by colors, drawn arrows indicate birth, dashed arrows the necessary shift.}\label{fig:lookdown_event}
\end{figure}
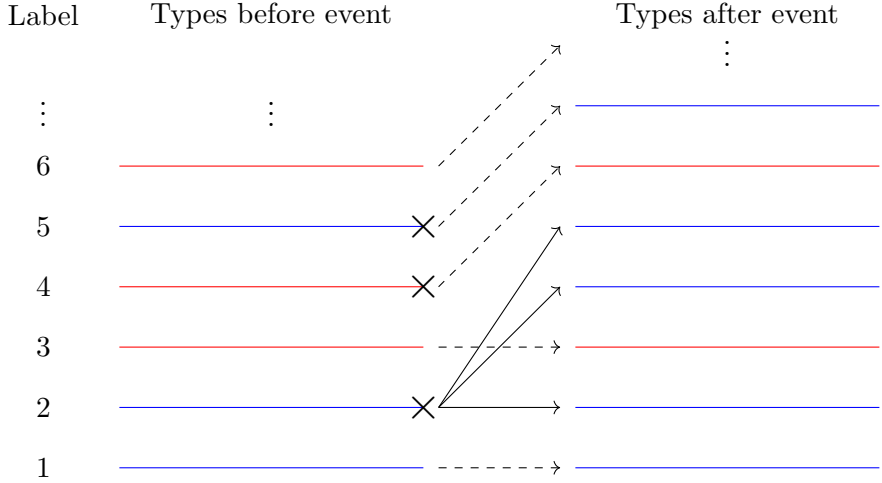

More rigorously, let $\xi$ be a Poisson point process on $[0,+\infty)\times \mathcal{Z}\times [0,1]^{\mathbb{N}}$ with mean intensity $\mathrm{d}s\otimes \Pi(\mathrm{d}\bm z)\otimes \mathrm{d}\mathsf{u}$, where $\mathrm{d}\mathsf{u}$ denotes the Lebesgue measure on $[0,1]^\mathbb{N}$.
Then, we construct the population size process as the strong solution to
\[
N_t = N_0 + \int_0^t \Big[ \gamma_b - \gamma_d N_s\Big] \d s + \int_{[0,t]\times \mathcal{Z}\times [0,1]^\mathbb{N}} \Big[ z_b - z_d N_{s-}\Big] \;\xi(\mathrm{d}s, \mathrm{d}\bm z, \mathrm{d}\mathsf{u}).
\]
The levels with label $n$ will not influence levels with label $m < n$, so that it is enough to describe the dynamics on the space $\mathbb{K}^n$ of $n$ levels, $n\in\mathbb{N}$.
Enumerate the points of $\xi = \big\{ (t_i, \bm z_i, \mathsf{u}_i) \big\}_i$ and write $\overline{\bm z}_i := \overline{\bm z}(N_{t_i-}, \bm z_i)$ for the corresponding effective impact.
For a set of levels $J\subseteq \lint 1, n\rint := \{1, \dots, n\}$ with $\vert J\vert\geq 2$, we denote by $L_J^n(t)$ the number of times that exactly these levels have been chosen in an event before time $t$:
\[
L_J^n(t) := \sum_{i:t_i \leq t} \prod_{j\in J} \mathds{1}_{\mathsf{u}_i(j)\leq \overline{\bm z}_i} \prod_{j\not\in J} \mathds{1}_{\mathsf{u}_i(j) > \overline{\bm z}_i}.
\]
For $\vert J\vert \leq 1$, we set $L_J^n \equiv 0$ as an event has no effect if at most one level is chosen to participate.
Since
\[
\mathbb{E}\Big[ L_J^n(t)\;\vert\; \{ (t_i, \bm z_i)\}_i\Big] = \sum_{i: t_i \leq t} \overline{\bm z}_i^{\vert J\vert}(1- \overline{\bm z}_i)^{n - \vert J\vert} \leq \sum_{i: t_i\leq t} \overline{\bm z}_i,
\]
\Cref{lem:finiteness_jumps_lookdown} below ensures that the processes $L_J^n$ a.s.~jump only finitely often in any compact time interval.
At a jump time of $L_J^n$, the vector $\mathsf{k}\in \mathbb{K}^n$ is replaced by $\theta_J^n(\mathsf{k})$, in which copies of $\mathsf{k}(\min J)$ are inserted at the indices in $J\setminus \{\min J\}$ and the last $\vert J\vert - 1$ entries are discarded to obtain a vector of length $n$:
\[
\theta_J^n(\mathsf{k})(i) = \begin{cases}
\mathsf{k}(i) &\text{ if } i < \min J\\
\mathsf{k}(\min J) & \text{ if } i\in J\\
\mathsf{k}\big(i - \vert J\cap \lint 1, i\rint\vert + 1\big) & \text{ otherwise}
\end{cases}.
\]
We set $\theta_J(\mathsf{k}) = \mathsf{k}$ for $\vert J\vert \leq 1$.

\begin{lem}\label{lem:finiteness_jumps_lookdown}
The sum $\sum_{i:t_i\leq t} \overline{\bm z}_i$ is a.s.~finite for every $t\geq 0$ .
\end{lem}
\begin{proof}
As in \Cref{theo:well-posedness_total_population}, let
\[
\tau_\epsilon := \inf\{t\geq 0\;:\; N_t\not\in U_\epsilon\text{ or } N_{t-}\not\in U_\epsilon\},\quad\text{ where } U_\epsilon = (\epsilon, 1/\epsilon).
\]
Then, for any $\epsilon > 0$
\begin{align*}
		\mathbb{E}\left[\sum_{i : t_i \leq t\wedge\tau_\epsilon } \overline{\bm{{z}}}_i \right] &\leq \mathbb{E}\left[\mathds{1}_{N_0\geq \epsilon} \int_{[0,t] \times \mathcal{Z}\times [0,1]^n} \frac{{z}_b}{(1 - {z}_d) \epsilon + {z}_b} \;\xi(\mathrm{d}t,\mathrm{d}\bm{z},\mathrm{d}u) \right]\\
		&\leq t\int_\mathcal{Z} \dfrac{{z}_b}{(1-{z}_d)\epsilon + {z}_b}\d\Pi(\bm{z})\\
		&\leq t\cdot \int_\mathcal{Z} \mathds{1}_{\{{z}_d\geq 1/2\}} + \dfrac{2{z}_b}{\epsilon}\d\Pi(\bm{z}) < +\infty.
	\end{align*}
Consequently,
\begin{align*}
\mathbb{P}\left(\sum_{i : t_i < t} \overline{\bm{z}_i} = +\infty \right) &\leq \mathbb{P}\left(\bigcap_{\epsilon \downarrow 0} \{\tau_\epsilon < t\}\right) = \lim_{\epsilon\downarrow 0} \mathbb{P}(\tau_\epsilon < t) = 0. \qedhere
\end{align*}
\end{proof}

The proof that this is indeed the correct lookdown process for the $(\bm\gamma,\Pi)$-Fleming-Viot process will be postponed to \Cref{ssec:well_posedness}.

\subsection{The Markov Mapping Theorem}

Before stating the theorem, we will give a short heuristic guide. 
Consider the following situation: on a base space $B$, we have a solution $Y$ to a given martingale problem $A_B$.
However, part of the stochastic information is not explicitly observed; one could think of the driving noise of an SDE or the level structure of a lookdown.
The Markov Mapping Theorem will allow us in specific cases to retrieve this additional information.
More precisely, we want to find a process $X$ on an \emph{extended} space $E$ such that it ``projects" to the initial process: $Y = \Gamma(X)$ in distribution for a measurable function $\Gamma: E\rightarrow B$.
This will be possible if the generator $A_B$ can be viewed as the averaged version $\alpha A_E$ of a generator $A_E$ of $E$, where the averaging $\alpha$ has to behave well w.r.t.~the projection $\Gamma$.
In our case, $B$ will represent the measure-valued population distribution, whereas $E$ will denote the space of labelled ``individuals" in a loose sense.
The projection $\Gamma$ will retrieve the distribution as the DeFinetti measure of the labelled population.

The following theorem will make these ideas rigorous. We will use the notation for multivalued generators from \cite{EK86} and the terminology of \cite{EK19}. 
Let $E$ be a Polish space, $B(E)$ be the space of bounded measurable functions and $C_b(E)$ be the space of bounded continuous functions.

\begin{theo}[Markov Mapping Theorem, {\cite[Corollary 3.2]{NK11}}]\label{theo:MMT}
Let $E$ and $B$ be separable Polish spaces. Let $A_E\subseteq C_b(E) \times B(E)$ and $\psi\in C(E)$ with $\psi \geq 1$. 
Suppose that for every $f\in \mathcal{D}(A_E)$, there exists a $c_f\geq 0$ such that
\begin{equation} \label{eq:mmtpsi}
\vert g(x)\vert \leq c_f \psi(x)\qquad\text{ for all }x\in E
\end{equation}
for all $g\in B(E)$ such that $(f,g)\in A_E$. 
Define $A_E^0f := \frac{A_E f}{\psi}$.

Suppose that $A_E^0$ is a countably determined pre-generator, and suppose that $\mathcal{D}(A_E^0) = \mathcal{D}(A_E)$ is closed under multiplication and is separating. 
Let $\Gamma: E\rightarrow B$ measurable, and let $\alpha$ be a transition kernel from $B$ to $E$ satisfying $\alpha\big(y, \Gamma^{-1}(y)\big) = 1$ for all $y\in B$. 
Furthermore, assume that $\alpha\psi < +\infty$ and define
\[
A_B := \left\{ \big(\alpha f, \alpha (A_Eg)\big)\;\vert\; (f,g)\in A_E\right\}.
\]
Let $\mu\in \mathcal{M}_1(B)$ and $\nu := \int_B \alpha(y,\cdot)\d\mu(y)$. 
Suppose $\tilde{Y}$ is a \emph{{càdlàg}} solution to the martingale problem for $(A_B,\mu)$ without fixed point of discontinuity.
\begin{enumerate}
	\item If $\tilde{Y}$ satisfies $\int_0^t \alpha\psi(\tilde{Y}_t) < +\infty$ a.s.~for all $t\geq 0$, there exists a solution $X$ to the martingale problem for $(A_E,\nu)$. 
	If $Y := \gamma\circ X$ is \emph{càdlàg}, then $Y$ and $\tilde{Y}$ have the same distribution on $D_B([0,+\infty))$.
	\item For any $t\geq 0$, one has $\mathbb{E}[f(X_t)\;\vert\; \mathcal{F}_t^Y] = \alpha f(Y_t)$.
	\item If uniqueness holds for the martingale problem for $(A_E, \nu)$, then uniqueness holds for the martingale problem for $(A_B, \mu)$.
\end{enumerate}
\end{theo}

\subsection{Well-posedness of the $($\texorpdfstring{$\bm{\gamma}$}{\gamma}$,\Pi)$-Fleming-Viot martingale problem}\label{ssec:well_posedness}

In this section, we link the lookdown process from \Cref{ssec:contruction_random_environment} back to the $(\bm\gamma,\Pi)$-Fleming-Viot process.
Well-Posedness of the martingale problem then follows from the Markov Mapping Theorem.

Here, the base space is given by $B := (0,+\infty)\times \mathcal{M}_1(\mathbb{K})$.
The extended space $E := (0,+\infty)\times \mathbb{K}^\mathbb{N}$ is the space of labelled particles.
We set the projection $\Gamma : E\rightarrow B$ to be the empirical average
\[
\Gamma(N, \mathsf{k}) := \left( N, \lim_{n\to+\infty} \dfrac{1}{n}\sum_{i=1}^n \delta_{\mathsf{k}(i)}\right)
\]
whenever it exists and set $\Gamma(N,\mathsf{k}) := (N,\rho_0)$ otherwise for some arbitrary, but fixed $\rho_0\in\mathcal{M}_1(\mathbb{K})$.
Conversely, we define the kernel $\alpha$ from ${B}$ to ${E}$ by
\[
\alpha(N, \rho; \cdot) := \delta_N \otimes \rho^{\otimes \mathbb{N}}.
\]

Next, we identify the martingale problem for the lookdown process.
Define the space $\mathcal{D}$ of test functions to contain $h:E\rightarrow\mathbb{R}$ of the form
\begin{equation}\label{eq:test_function_lookdown}
h(N,\mathsf{k}) = f(N)\prod_{i=1}^n g_i\big(\mathsf{k}(i)\big)
\end{equation}
for $f\in C_b^1(\mathbb{R})$, $n\in\mathbb{N}$, and $g_1,\dots,g_n$ continuous with $0\leq g_i\leq 1$.
Then, the lookdown process is solution to the martingale problem for
\begin{align*}
Ah(N,\mathsf{k}) &= \big( \gamma_b - \gamma_d N\big) \dfrac{f'(N)}{f(N)}h(N, \mathsf{k})\\
&\qquad +\int_\mathcal{Z} \left(\sum_{J\subseteq \lint 1, n\rint} \overline{\bm{z}}^{\vert J\vert}\left(1-\overline{\bm{z}}\right)^{n-\vert J\vert} h\Big( (1-{z}_d)N + {z}_b, \theta_J(\mathsf{k})\Big)  - h(N,\mathsf{k})\right)\d{\Pi(\bm{z})},
\end{align*}
defined on $\mathcal{D}$, where $h$ is assumed to be given by \eqref{eq:test_function_lookdown}.

To compute the averaged generator $\alpha A$, we first note that
\[
\alpha h(N, \rho) = f(N) \prod_{i=1}^n \overline{g}_i,
\]
where $\overline{g}_i := \int_\mathbb{K} g_i(\kappa)\d{\rho(\kappa)}$. 
If we apply the kernel to the lookdown generator $A$, this gives
\begingroup
\allowdisplaybreaks
\begin{align*}
&\alpha (A h)(N,\rho) \\
&= (\gamma_b - N\gamma_d)\dfrac{f'(N)}{f(N)}\alpha h(N,\rho)\\
&\begin{aligned}
&\qquad + \int_\mathcal{Z} \Bigg[ f\Big( (1-{z}_d)N + {z}_b\Big)\\
&\qquad\qquad\qquad\qquad \times \sum_{J\subseteq \lint 1,n\rint} \overline{\bm{z}}^{\vert J\vert}\left(1 - \overline{\bm{z}}\right)^{n - \vert J\vert}  \left(\int_\mathbb{K} \prod_{i\in J} g_i(\kappa^*)\d{\rho(\kappa^*)}\right)\cdot \left(\prod_{i\not\in J} \overline{g}_i\right) \\
&\hspace{11cm} \vphantom{\sum_{J\subseteq \lint 1,n\rint} \overline{\bm{z}}^{\vert J\vert}\left(1 - \overline{\bm{z}}\right)^{n - \vert J\vert}}- \alpha h(N,\rho)\Bigg] \d{\Pi(\bm{z})}
\end{aligned}\\
&= (\gamma_b - N\gamma_d)\dfrac{f'(N)}{f(N)}\alpha h(N,\rho)\\
&\qquad + \int_\mathcal{Z}\int_\mathbb{K} \left[ f\Big(1-{z}_d)N + {z}_b\Big) \prod_{i=1}^n\Big( (1-\overline{\bm{z}})\overline{g}_i + \overline{\bm{z}}g_i(\kappa^*)\Big) - \alpha h(N,\rho) \right]\d{\rho(\kappa^*)}\d{\Pi(\bm{z})}\\
&= (\gamma_b - N\gamma_d)\dfrac{f'(N)}{f(N)}\alpha h(N,\rho)\\
&\qquad + \int_\mathcal{Z} \int_\mathbb{K}\alpha h\left( (1-{z}_d)N + {z}_b, \left(1 - \overline{\bm{z}}\right)\rho + \overline{\bm{z}}\delta_{\kappa^*}\right)\d{\rho(\kappa^*)}\d{\Pi(\bm{z})},
\end{align*}
\endgroup
which is the generator of the $(\bm\gamma,\Pi)$-Fleming-Viot martingale problem, verifying the basic assumption of the Markov Mapping Theorem \ref{theo:MMT}.

We will divide the prove of \Cref{theo:well-posedness} into three parts: we first prove that the martingale problem is regular, then existence and uniqueness.

\begin{lem}\label{lem:regularity}
Consider a $(\bm\gamma,\Pi)$-Fleming-Viot process. 
Then:
\begin{enumerate}
	\item there exists a \emph{càdlàg} modification of the process;
	\item this modification is quasi-left continuous and thus has no fixed point of discontinuity.
\end{enumerate}
\end{lem}
\begin{proof}[Proof of \Cref{lem:regularity}]
Note that ${B} = (0,+\infty)\times \mathcal{M}_1(\mathbb{K})$ is separable and that $C_b^1({B})$ is separating and contains a countable set that separates points. If $(N,\rho)$ is a $(\bm\gamma, \Pi)$-Fleming-Viot, then $N$ is the population size process studied in \Cref{sec:populationsize} and thus
\[
\mathbb{P}\Big( (N_t,\rho_t)\in \Gamma_{T,\epsilon}\times \mathcal{M}_1(\mathbb{K})\text{ for all }t\in [0,T]\Big) \geq 1 - \epsilon
\]
for all $T,\epsilon > 0$ and a compact $\Gamma_{T,\epsilon}$ given by \Cref{cor:compact_containment_N}.
Hence, \cite[Theorem 4.3.6]{EK86} yields the existence of a \emph{càdlàg} modification.
Finally, \cite[Theorem 4.3.12]{EK86} ensures that the \emph{càdlàg} modification is quasi-left continuous, and thus has no fixed point of discontinuity.
\end{proof}

Since any modification of a solution to the martingale problem is a solution itself, we may and will assume in the following that any solution is \emph{càdlàg}.

\begin{lem}\label{lem:existence}
For any given initial condition $(N_0, \rho_0)\in S$, there exists a $(\bm\gamma,\Pi)$-Fleming-Viot process.
\end{lem}
\begin{proof} 
	Existence is immediate as soon as the jump measure $\Pi$ is finite, as a solution may be constructed as a Piecewise Deterministic Markov Process. For a general characteristic $(\bm\gamma, \Pi)$, set $\Pi^{(n)}(\bm z) := \mathds{1}_{\Vert \bm z\Vert > 1/ n}\Pi(\bm z)$ and choose $\bm\gamma^{(n)}$ such that $(\bm\gamma^{(n)},\Pi^{(n)})\in\mathfrak{C}$ and $\bm\gamma^{(n)} \to \bm\gamma$. Then, $(\bm\gamma^{(n)},\Pi^{(n)})\to (\bm\gamma,\Pi)$ in the sense of characteristics and we may apply \Cref{theo:closedness} to obtain a $(\bm\gamma,\Pi)$-Fleming-Viot process as the limit of $(\bm\gamma^{(n)},\Pi^{(n)})$-Fleming-Viot processes, see also \Cref{sec:closedness}.
	\qedhere
\end{proof}

By \cite[Theorems 4.4.2 and 4.4.6]{EK86} the strong Markov property follows for both the lookdown and projected process from the next two lemmas.
\begin{lem}[{see \emph{e.g.}~\cite[Section 4]{DK99}}]\label{lem:uniqueness-lookdown}
The martingale problem for $A$ is well-posed.
In particular, the unique solution is given by the construction in \Cref{ssec:contruction_random_environment}.
\end{lem}
\begin{proof}
The main idea is reproduced for the convenience of the reader.
First, consider the localizing sequence 
\[
\tau_\epsilon := \inf \{ t\geq 0\;:\; N_t \not\in U_\epsilon \text{ or }N_{t-}\not\in U_\epsilon\}\quad\text{ with }\quad U_\epsilon := (\epsilon, 1/\epsilon)
\] 
from \Cref{theo:well-posedness_total_population}.
Then, the stopped martingale problem $(A, \tau_\epsilon)$ has bounded jump rates and we may conclude uniqueness.
Finally, the result is extended to the martingale problem for $A$ by virtue of \cite[Theorem 4.6.2]{EK86} and the fact that
\[
\lim_{\epsilon \downarrow 0}\mathbb{P}\Big( \tau_\epsilon \leq t \Big) = 0
\]
for all $t\geq 0$.
\end{proof}

\begin{lem}\label{lem:uniqueness}
The $(\bm\gamma,\Pi)$-Fleming-Viot martingale problem is well-posed for any $(\bm\gamma,\Pi)\in\mathfrak{C}$.
\end{lem}
\begin{proof}
In order to apply \Cref{theo:MMT}, it suffices to find a function $\psi$ satisfying \eqref{eq:mmtpsi} and $\int_0^t \alpha \psi (\tilde{Y}_t) < + \infty$.	We have the bound
	\begin{align}
		\vert Ah(N, \kappa) \vert &\leq (\gamma_b + N\gamma_d) \Vert f'\Vert_\infty\nonumber\\
		&\qquad + \prod_{i=1}^n g_i(\kappa_i) \int_\mathcal{Z} f\Big( (1- {z}_d)N + {z}_b\Big) - f(N)\d{\Pi(\bm{z})}\nonumber\\
		&\qquad + 2\Vert h\Vert_\infty\sum_{\substack{J\subseteq \lint 1, n\rint\\ \vert J\vert \geq 2}} \int_\mathcal{Z} \overline{\bm{z}}^2 \d{\Pi(\bm{z})}\nonumber\\
		&\leq C\cdot N + C2^n\cdot\left(1 + \dfrac{1}{N}\right)\label{eq:lookdown_psi_bound}
	\end{align}
	for a constant depending on the choice of $h$ and the characteristic $(\bm\gamma, \Pi)$. 
	Here, we used that $\overline{\bm{z}} = \frac{{z}_b}{(1-{z}_d)N + {z}_b}$ and thus that $\overline{\bm{z}}^2 \leq \overline{\bm{z}}$ is bounded by $2{z}_b N^{-1}$ if ${z}_d \geq 1/2$ and by $1$ otherwise.
Let $(N,\rho)$ be a $(\bm\gamma,\Pi)$-Fleming-Viot process and take the localizing sequence $(\tau_\epsilon)_{\epsilon\downarrow 0}$ from the previous proof.
One has
\[
\mathbb{E}\left[\int_0^{t\wedge\tau_\epsilon} N_s + \dfrac{1}{N_s}\d{s}\right] < 2t\epsilon^{-1} <+\infty
\]
for all $t\geq 0$.
The same argument as in the proof of \Cref{lem:finiteness_jumps_lookdown} yields that
\[
C\int_0^t 1 + N_s + \dfrac{1}{N_s} \d s 
\]
is a.s.~finite for every $t\geq 0$.
Using the bound \eqref{eq:lookdown_psi_bound}, point 1.~of \Cref{theo:MMT} ensures the existence of a solution to the lookdown martingale problem $(A, \alpha(N_0, \rho_0;\cdot))$.
By \Cref{lem:uniqueness-lookdown} and point 3.~of \Cref{theo:MMT}, we conclude that the $(\bm\gamma,\Pi)$-Fleming-Viot martingale problem has at most one solution for any initial condition.
\end{proof}

In addition to the well-posedness of the $(\bm\gamma,\Pi)$-Fleming-Viot martingale problem, the application of the Martingale Mapping Theorem yields information on the lookdown process.

\begin{corol}\label{corol:exchangeability}
Let $(N,\mathsf{k})$ be the lookdown process from \Cref{ssec:contruction_random_environment} corresponding to the $(\bm\gamma,\Pi)$-Fleming-Viot process $(N,\rho)$. Then, $\mathsf{k}_t$ is an exchangeable sequence with DeFinetti measure $\rho_t$ for every $t\geq 0$.
\end{corol}
\begin{proof}
This is point 2.~of \Cref{theo:MMT}.
\end{proof}

\section{Genealogy and duality} \label{sec:duality}

One can only refer to a coalescent of an infinite-population limit as the genealogy, if the coalescent arises as the limit of the genealogies of the corresponding individual-based models. We will conclude this property for our model from the lookdown framework.
Then, we will discuss the implications of the varying, non-reversible population size of the $(\bm{\gamma}, \Pi)$-Fleming-Viot process for the notion of duality between the coalescent and the proportion process.
Finally, we will investigate fixation of types and whether the associated coalescent comes down from infinity.

\subsection{Individual-based model} \label{subsec:genealogy}

Fix a characteristic $(\bm\gamma,\Pi)$, an initial condition $N_0 > 0$ and $m\in\mathbb{N}$ which should be thought of as \emph{carrying capacity} of the population.

Heuristically, we will consider the following model: we start with an initial population $N_0^m := \lfloor mN_0\rfloor$ and impose two dynamics on the individuals:
\begin{description}
	\item[Continuous part:] Each individual carries an independent Poisson clock with rate $\gamma_d$.
	When the clock rings, the individual dies.
	Independently, there is a Poisson clock with rate $m\gamma_b$.
	Whenever it rings, an individual is chosen at random to produce an offspring.
	
	\item[Discrete events:] Reproduction events fall as a Poisson point process on $[0,+\infty)\times \mathcal{Z}$ with mean intensity $\mathrm{d}t\otimes \Pi(\mathrm{d}\bm z)$.
	At an event $(t,z)$, $\lfloor N_t^m z_d\rfloor$ individuals die and $\lfloor mz_b\rfloor$ individuals are born and carry the type of an individual chosen uniformly at random from the population immediately before the event.
\end{description}
Events with $z_d < 1/N_t^m$ and $z_b < 1/m$ do not appear at all, so only a finite number of events occur on any compact time interval.
In the following, we will assume the process to be stopped at the first time the population size $N_t^m = 0$ reaches zero.

To fit the above into the setting of \cite{DK99}, we decompose the population process $N^m_t = N^m_0 + B^m_t - D^m_t$ into births and deaths. The dynamics are described by the generator
\begin{align*}
\mathcal{L}^mf(n, b, d) &= m\gamma_b\Big( f(n+1, b+1, d) - f(b,d)\Big) + n\gamma_d\Big( f(n-1, b, d+1) - f(b, d)\Big) \\
&\qquad+ \int_\mathcal{Z} \Big( f( n - \lfloor nz_d\rfloor + \lfloor m z_b\rfloor, b + \lfloor mz_b\rfloor, d + \lfloor nz_d\rfloor) - f(n,b,d)\Big) \d\Pi(\bm z).
\end{align*}
The first part is a continuous birth and death generator, whereas the second part implements reproduction events.

Similarly to the infinite population limit, we can construct this population model as a lookdown process, see \cite[Section 1.2]{DK99}.
More precisely, we start with an exchangeable sequence $\mathsf{k}^m_0 \in \mathbb{K}^m$.
If the population changes at time $t$, we denote by $k := \Delta B^m_t$ the number of births.
Births work exactly as in the lookdown construction from \Cref{ssec:contruction_random_environment}: first, choose $k+1$ levels $i_0 < \dots < i_{k}$ randomly from $1,\dots, N^m_t$; then, insert individuals of the type $\mathsf{k}_{t-}^m(i_0)$ at levels $i_1,\dots,i_k$ and move all other individuals upwards.
All individuals with level greater than $N^m_t$ are considered dead, which automatically takes care of death events.

We wish to prove that this lookdown construction converges in a suitable sense to the one defined in \Cref{ssec:contruction_random_environment}.
One verifies immediately that the the sequence $(N^m/m)_{m\in\mathbb{N}}$ converges in law to the population process $N$ started from $N_0$.
However, to apply \cite[Theorem 3.2]{DK99}, we need to prove a stronger convergence.
More precisely, writing $[ B^m]$ for the quadratic variation process of $B^m$, we define
\[
U^m_t := \dfrac{[ B^m]_t + B^m_t}{m^2}.
\]
One then checks that $(N^m/m, U^m)$ converges weakly to $(N, U)$, where
\[
U_t = \int_{[0,t]\times \mathcal{Z}} z_b^2 \;\xi(\mathrm{d}t, \mathrm{d}\bm z)
\]
if $\xi$ denotes the driving noise of $N$ (see \Cref{theo:well-posedness_total_population}).
Using the fact that $N$ does not reach $0$ in finite time together with the remark just after Equation (3.12) in \cite{DK99}, we deduce the following result on genealogies.

\begin{theo}[see {\cite[Theorem 3.2]{DK99}}] \label{theo:genealogy}
Let $\mathsf{k}_0\in \mathbb{K}^\mathbb{N}$ be an exchangeable sequence with DeFinetti measure $\rho_0$ and set $\mathsf{k}^m_0 := (\mathsf{k}_0(i))_{1\leq i\leq m}$.
Consider the lookdown process $(\mathsf{k}^m_t)_{t\geq 0}$ defined above and view it as a process on $\mathbb{K}^\mathbb{N}$ by setting $\mathsf{k}^m_t(i) := \mathsf{k}^m(i)$ if $\max_{s\leq t} N^m_s \leq i$ and $\mathsf{k}^m_t(i) := \mathsf{k}^m_{\tau^m_i(t)-}(i)$, where $\tau^m_i(t) := \sup\{ s< t\;:\; N^m_s \geq i\}$ otherwise.

Write $(N_t, \mathsf{k}_t)_{t\geq 0}$ for the lookdown process from \Cref{ssec:contruction_random_environment} and let $\rho_t$ be the DeFinetti measure of $\mathsf{k}_t$.
Then, setting
\[
\rho^m_t := \dfrac{1}{N^m_t}\sum_{i=1}^{N^m_t} \delta_{\mathsf{k}^m_t(i)}
\]
up to the extinction time, the sequence $\left(\frac{N^m}{m}, \frac{N^m}{m}\rho^m, \mathsf{k}^m\right)$ converges in law to $(N, N\rho, \mathsf{k})$.
In particular, the genealogy of $\mathsf{k}$ approximates the genealogy of the individual-based model.
\end{theo}

\subsection{Duality} \label{subsec:duality}

The notion of duality that we will apply becomes slightly more subtle due to the changes in population size.
For this reason, we will give a short reminder of different duality concepts.

\begin{defin}[See \emph{e.g.}~{\cite[Definition 1.1]{JK14}}]
Consider two Markov processes $X = \Big(\Omega_1, \mathcal{F}_1, (X_t)_{t\in[0,T]}, (\mathbb{P}_x)_{x\in E}\Big)$ and $Y = \Big(\Omega_2, \mathcal{F}_2, (Y_t)_{t\in[0,T]}, (\mathbb{P}^y)_{y\in F}\Big)$ on state spaces $E$ and $F$.
We say that $Y$ is the \emph{dual} to $X$ w.r.t.~a bounded measurable $H: E\times F\rightarrow\mathbb{R}$ if
\begin{equation}\label{eq:duality}
\mathbb{E}_x[H(X_t, y)] = \mathbb{E}^y[H(x, Y_t)]
\end{equation}
for all $x\in E$, $y\in F$ and $t\in [0,T]$.
In this case, $H$ is referred to as \emph{duality function}.
\end{defin}

Often, duality can be checked via a simple generator argument. If $L_X$ and $L_Y$ denote the time-homogeneous generator of $X$ and $Y$ and if $H$ is regular enough, duality w.r.t.~$H$ follows from
\[
\big(L_X H(\cdot, y)\big)(x) = \big( L_Y H(x,\cdot)\big)(y)
\]
for all $x\in E$ and $y\in F$, see {e.g.}~\cite[Proposition 1.2]{JK14} for details.
If a lookdown construction is available, duality can often be strengthened to \emph{pathwise duality}, where the two processes are constructed on the same probability space and \Cref{eq:duality} holds almost surely.

\begin{defin}[See \emph{e.g.}~{\cite[Definition 4.2]{JK14}}] \label{def:pathwiseduality}
Let $X, Y$ and $H$ be as before. Suppose that there exists families $\{ (X_t^x)_{t\in [0,T]}\}_{x\in E}$ and $\{ (Y_t^y)_{t\in[0,T]} \}_{y\in Y}$ on a common probability space $(\Omega,\mathcal{F},\mathbb{P})$ such that:
\begin{enumerate}[label=\roman*)]
	\item for all $x\in E$ and $y\in Y$, the finite-dimensional distributions of $X^x$ and $Y^y$ under $\mathbb{P}$ agree with those of $X$ and $Y$ under $\mathbb{P}_x$ and $\mathbb{P}^y$ resp.;
	\item for all $t\in [0,T]$, $x\in E$ and $y\in F$, we have
	\[
	H(x, Y_T^y) = H(X_t^x, Y_{T-t}^y) = H(H_T^x, y)\qquad\text{ $\mathbb{P}$-a.s.}.
	\]
\end{enumerate}
Then, $X$ and $Y$ are said to be \emph{pathwise dual} w.r.t.~$H$.
\end{defin}

Let us consider the graphical construction from \Cref{ssec:contruction_random_environment}. 
We will focus on the case $\mathbb{K} = \{a, A\}$ of two types and discuss the extension to more general type spaces at the end of the section.
For any $x\in [0,1]$, we can start the process from the initial condition $\rho_0 = x\delta_a + (1-x)\delta_A$ and track the proportion of one type $X_t^x := \rho_t(\{a\})$.
Then, we can fix a time horizon $T > 0$ and we can trace the $y \in \mathbb{N}$ lowest levels of the lookdown construction backwards in time by following the arrows from tip to tail.
Let $Y_t$ denote the number of these ancestral lineages at time $t$ backwards in time (or $(T-t)-$ forward in time).
Then, as all individuals are of the same type if and only if their ancestors in the lookdown were of the same type, we have
\begin{equation}\label{eq:kinda_pathwise}
x^{Y_T^\ell} = \left(X_t^x\right)^{Y_{T-t}^y} = \left(X_T^x\right)^\ell
\end{equation}
almost surely. 
However, this is \emph{not} a pathwise duality in the sense of \Cref{def:pathwiseduality} as the processes $X$ and $Y$ are not Markovian. 
Furthermore, this cannot be salvaged by forcing Markovianity by adding more information: in order for $Y$ to be Markov, it is necessary to add some information from the graphical construction \emph{backwards-in-time}.
Since the time reversed population process cannot start from an arbitrary initial value as it is constrained by the final value forward-in-time, this will not induce a pathwise duality.

Instead, recall that the lookdown process was constructed via the driving Poisson point process $\xi$ on $[0,+\infty)\times \mathcal{Z}\times [0,1]^\mathbb{N}$ with mean intensity $\mathrm{d}t\otimes \Pi(\bm z)\otimes \mathrm{d}\mathsf{u}$.
Writing $\xi_0$ for the projection of $\xi$ onto $[0,+\infty)\times \mathcal{Z}$, we may view the lookdown process as being constructed in the random environment $\xi_0$ which determines the times and strength of the impacts.
Let $(\ascnode,\mathscr{F},\mathscr{P})$ denote the probability space supporting $\xi$, and write $\mathscr{P}_{\xi_0}$ for the regular conditional probability of $\mathscr{P}$ with respect to $\xi_0$.
If $\mathsf{P}$ denotes the law of the random environment $\xi_0$, then 
\[
\mathscr{E}[f] = \mathsf{E}\Big[ \mathscr{E}_{\xi_0}[f] \Big].
\]
The laws $\mathscr{P}$ and $\mathscr{P}_{\xi_0}$ are usually referred to as the \emph{annealed} and the \emph{quenched} law respectively.

Under $\mathscr{P}_{\xi_0}$, $X$ and $Y$ are inhomogeneous Markov processes .
However, since $\xi_0$ also prescribes the times of possible jumps, the processes $X$ and $Y$ cannot be weakly continuous anymore and there will be times at which the sets $\{ X_t \neq X_{t-} \}$ and $\{ Y_t \neq Y_{t-}\}$ have positive probability.
This immediately implies that \Cref{eq:kinda_pathwise} can only hold for every $t\in [0,T]$ a.s.~under the quenched law if and only if we choose the left-continuous variant of $Y$, which is \emph{not} a modification.
Indeed, $X$ and the left-continuous variant of $Y$ are pathwise dual under the quenched law with respect to the usual moment duality functions $H(x,y) = x^y$.

The dependence of the dual process on the driving noise makes it difficult to describe it in a useful way.
Using the framework from \cite{martingale_problem_revisited}, one option is to view it as a solution to a \emph{canonical abstract martingale problem} with control variable $\xi_0$.
More precisely, all processes of the form
\begin{align*}
M^f_t &:= f(Y_t^y) - f(y) \\
&\qquad- \int_{[T-t,T]\times \mathcal{Z}} \sum_{k=2}^{Y_s^y} \left(\dfrac{z_b}{N_s}\right)^k\left(1 - \dfrac{z_b}{N_s}\right)^{Y_s^y - k} \cdot\Big( f(Y_s^y - k + 1) - f(Y_s^y)\Big) \;\xi_0(\mathrm{d}s, \mathrm{d}\bm z)
\end{align*}
for bounded $f:\mathbb{N}\rightarrow\mathbb{R}$ are martingales under the quenched law of $Y^y$ w.r.t.~to the filtration generated by the past of $Y$ and the complete information from $\xi_0$.
Note that $N_s$ is a measurable function of $\xi_0$ and does not provide any additional randomness, so that $M^f_t$ is a measurable function of $(Y^y, \xi_0)$ for every $t\in[0,T]$.
In the following, however, we will concentrate on a forward-in-time approach to derive genealogical properties.

We end this section with a short discussion of the case of general type spaces $\mathbb{K}$. 
In this case, the lineage counting process is not enough, but the argument still applies when considering the partition-valued process $(\pi_t^y)_{t\in[0,T]}$ tracking the whole coalescent of the lowest $y$ levels backwards in time.
More precisely, we start with the trivial partition $\pi_0^y := \{ \{1\}, \dots, \{y\}\}$. Every time the lineage counting process $Y^y$ jumps by $\Delta Y_t^y = l$, one chooses $l$ blocks from $\pi_{t-}^y$ and merges them.
For a function $f:\mathbb{K}^y \rightarrow \mathbb{R}$ and a partition $\pi = \{ A_1,\dots, A_m\}$ of $\{1,\dots,y\}$, we then define $f_\pi: \mathbb{K}^{\pi}\rightarrow\mathbb{R}$ to be the function 
\[
f_\pi(k_{A_1},\dots, k_{A_{m}}) = f(k_{B_1}, \dots, k_{B_y}),
\]
where $B_j \in \pi$ is the unique bloc in $\pi$ containing $j$.
Then, the non-Markovian duality from \Cref{eq:kinda_pathwise} can be extended to
\[
\int_{\mathbb{K}^{\pi_T^y}} f_{\pi_T^y}(\mathsf{k}) \;\rho_0^{\otimes \pi_T^y}(\mathrm{d}\mathsf{k}) = \int_{\mathbb{K}^{\pi_{T-t}^y}} f_{\pi_{T-t}^y}(\mathsf{k}) \;\rho_t^{\otimes \pi_{T-t}^y}(\mathrm{d}\mathsf{k}) = \int_{\mathbb{K}^y} f(\mathsf{k})\;\rho_T^{\otimes y}(\mathrm{d}\mathsf{k})\qquad\text{a.s.}
\]
for all continuous bounded $f:\mathbb{K}^y\to\mathbb{R}$.
As before, pathwise duality also holds in the quenched version, at least if we consider the left-continuous coalescent.

\begin{rem} \label{rem:dualityuniqueness}
	The pathwise duality provides us with an alternative proof of uniqueness that avoids treating uniqueness of the lookdown process. This might be a useful observation in other situations, where the uniqueness of the lookdown construction is difficult, see \cite[Remark 4.9]{EK19}.
	
	Suppose $(N,\rho)$ and $(\hat N, \hat{\rho})$ are two solutions for the $(\bm{\gamma}, \Pi)$-Fleming-Viot martingale problem. The Markov Mapping Theorem \ref{theo:MMT} and the lookdown construction yields two lookdown processes $(N, \mathsf{k})$ and $(\hat{N}, \hat{\mathsf{k}})$ and even their respective driving Poisson point processes $\xi$ and $\hat{\xi}$; see \cite{EquivSDEMart} for the reconstruction of the driving noise.
	By characterizing all moments, the conditional pathwise duality from above uniquely determines the laws of $\rho$ and $\hat{\rho}$ conditioned on $\xi_0$ respectively $\hat{\xi}_0$. 
	As the two Poisson point process have the same law, also the unconditional laws of $\rho$ and $\hat{\rho}$ are equal.
	
	Although we did not need to consider the uniqueness of the lookdown process, the existence of a corresponding lookdown is crucial in the argument.
	Indeed, uniqueness is not immediately implied by \Cref{eq:kinda_pathwise} as duality only provides uniqueness of the one-dimensional distributions.
	To extend this to all finite-dimensional distributions, one generally makes use of the Markov property.
\end{rem}

\subsection{Fixation}
We will first prove that the lowest levels of the lookdown process will eventually all be of the same type and deduce afterwards the equivalence to fixation of a certain type in the population.
\PropositionFixation*
\begin{proof}
	For two constants $0 < C_1 < C_2$ to be determined later and $i \geq 1$, we define $ \sigma_0 := 0$ and the down- and upcrossing stopping times
	\begin{align*}
		\tau_i &:= \inf \{t \geq \sigma_{i-1} \;:\; N_t \leq C_1 \text{ or } N_{t-}\leq C_1 \}, \\
		\sigma_i &:= \inf \{ t \geq \tau_i \;:\; N_t \geq C_2 \text{ or }N_{t-}\geq C_2 \}.
	\end{align*}
	Due to the ergodicity, a stationary distribution $\mu$ exists, satisfying in particular
	\begin{equation*}
		\frac{1}{T}\int_0^T \mathds{1}_{[0, C_1]} (N_t) \d t \underset{T\to+\infty}{\longrightarrow} \mu ([0, C_1]) \enspace a.s.
	\end{equation*}
	Choose $C_1$ large enough such that $\mu([0,C_1]) > 0$. 
	Then, the above implies that $\tau_1$ is a.s.~finite. 
	Assume for the moment that $\sigma_1$ is a.s.~finite. If $C_2\to +\infty$, then so does $\sigma_1$. 
	For any fixed $T_0>0$ and $\epsilon\in(0,1)$, we may choose $C_2 = C_2(T_0,\epsilon)$ large enough so that $\sigma_1 - \tau_1 > T_0$ with probability at least $1-\epsilon$. 
	Enumerate the point process $\xi = \{(t_i, \bm{z}^i, u^i)\}$ from the lookdown construction. 
	Then,
	\begin{align*}
	&\mathbb{P}\big(\text{There is a point }(t_i, \bm{z}^i, u^i)\text{ with ${z}^i_b \geq \delta$ in the time interval }[\tau_1,\sigma_1]\big) \\
	&\geq\mathbb{P}\Big(\text{There is a point }(t_i, \bm{z}^i, u^i)\text{ with ${z}^i_b\geq \delta$ in the time interval }[\tau_1,\tau_1+T_0]\\
	&\hspace*{1cm}\text{ and }\sigma_1 - \tau_1>T_0\Big)\\
	&= 1 - \mathbb{P}(\text{There is no such point in }[\tau_1,\tau_1 + T_0]) - \mathbb{P}(\sigma_1 - \tau_1 \leq T_0)\\
	&\geq 1 - \exp\left(- T_0\Pi_b\big([\delta,+\infty)\big)\right) - \epsilon.
	\end{align*}
	As $\Pi_b\neq 0$, we may choose $\delta> 0$ small enough such that $\Pi_b([\delta,+\infty))$ strictly positive. 
	Then, choose $T_0$ large enough that the above remains bounded away from $0$. 
	Finally, by \Cref{theo:well-posedness}, the strong Markov property extends this argument to the interval $[\tau_2, \sigma_2]$ etc. 
	Hence, there will be a.s.~infinitely many reproduction events with effective impact
	\[
	\overline{\bm{z}}^i = \dfrac{{z}_b^i}{(1- {z}_d^i)N_{t_i-} + {z}_b^i} \geq \dfrac{\delta}{C_2 + \delta} =: C > 0.
	\]
	Consequently, the $n$ lowest lineages are all effected by a common reproduction event with probability at least $C^n > 0$ a.s.~infinitely often and we may conclude that the $n$ lowest lineages are effected by a common reproduction event at an a.s.~finite time.
	
 Now, if $\sigma_i = +\infty$ with positive probability, we may proceed in a more straightforward way. 
 Indeed, this will imply that the interval $[\tau_i,\sigma_i]$ is infinite on this event, and thus contains infinitely many reproduction events satisfying $\overline{\bm{z}}^i \geq C$. 
 Hence, it suffices to choose $\epsilon := 1 - \mathbb{P}(\sigma_i = +\infty)$ and $T_0$ large enough to compensate.
\end{proof}

\begin{corol}\label{corol:fixation}
Suppose that $\mathbb{K} \subseteq [0,1]$ is closed. For a $(\bm\gamma,\Pi)$-Fleming-Viot process, let $F_t(\kappa) := \rho_t([0,\kappa])$ denote the cumulative distribution function of $\rho_t$. 
Under the assumptions of \Cref{pro:fixation}, $F_t(\kappa)$ converges for every $\kappa\in\mathbb{K}$ a.s.~to a random variable $F_\infty(\kappa)$ supported on $\{0,1\}$ such that
\[
\mathbb{P}\left(\left.\lim_{t\to+\infty} F_t(\kappa) = 1\;\right\vert\; F_0\right) = F_0(\kappa).
\]
\end{corol}
\begin{proof}
Without loss of generality, we may consider $\mathbb{K} = \{0,1\}$, as $F_t(\kappa)$ splits the type space into two sets. 
We will write $w_t := \rho_t(\{1\})$ and it suffices to show that $w_t \to w_\infty$ a.s.~for a Bernoulli random variable $w_\infty$ with conditional mean $w_0$, which we may choose to be deterministic. For a test function $f = f(w)$ independent of $N$, we obtain
\[
\mathcal{L}^{\bm\gamma, \Pi} f(N, w) = \int_\mathcal{Z} \left[ w f\Big( (1-\overline{\bm z})w + \overline{\bm z}\Big) + (1-w)f\Big((1-\overline{\bm z})w\Big) - f(w)\right]\d\Pi(\bm z).
\]
In particular, $w_t$ is a bounded martingale and we may conclude that it converges almost surely to its limit $w_\infty$. Additionally,
\[
w_t^n = \left(\int_\mathbb{K} \kappa \d\rho_t(\kappa)\right)^n = \mathbb{E}\left[\left. \prod_{i=1}^n \kappa_t(i)\;\right\vert\;\rho_t\right]\quad\text{a.s.}
\]
Hence, if we write $\tau_n$ for the almost surely finite time at which the $n$ lowest lineages are all descendants from the lowest lineage, then
\begin{align*}
w_{\tau_n}^n = \mathbb{E}\left[ \left.\prod_{i=1}^n \kappa_{\tau_n}(i)\;\right\vert\;\rho_t\right] = \mathbb{E}\left[\prod_{i=1}^n \kappa_0(0)\right] = \mathbb{E}[ \kappa_0(0) ] = w_0.
\end{align*}
This implies for all $n\in\mathbb{N}$
\begin{align*}
\mathbb{E}[w_\infty^n] &= \lim_{t\to+\infty} \mathbb{E}[ w_t^n] \\
&= \lim_{t\to+\infty} \Big( \mathbb{E}\left[w_{t}^n\mathds{1}_{\tau_n \leq t}\right] + \mathcal{O}(\mathbb{P}(\tau_n > t)\Big)\\
&= w_0\lim_{t\to+\infty} \mathbb{P}(\tau_n \leq t) = w_0,
\end{align*}
which concludes the proof.
\end{proof}

\subsection{Not coming down from infinity}
Coming down from infinity addresses the question whether a coalescent started from infinitely many lineages coalesces quickly enough to later have only finitely many lineages. The next result shows that our coalescent contains a non-trivial fraction of singleton lineages at any time, a situation know as containing dust.
This complements the previous result as it shows that fixation cannot occur in bounded time. A related discussion can be found on \cite[p.196]{DK99}.

\PropositionNotComingDown*
\begin{proof}
Let $t\geq 0$ be arbitrary.
Since the rate $\Pi_d(\{1\})$ at which the entire population is replaced is finite, there is a positive probability that no such event takes place before time $t$.
As such, we may and will assume $\Pi_b(\{1\}) = 0$.

Conditionally on the events $\xi_0 = \{ (t_i,\bm z_i) \}$, the first lineage does not jump until time $t$ with probability at least $\prod_{t_i\leq t} \left(1 - \overline{\bm{z}}_i\right)$, where $\overline{\bm{z}}_i := \frac{{z}_{b}^i}{(1-{z}_{d}^i)N_{t_i} + {z}_{b}^i}$. 
Now,
\begin{equation*}
\prod_{t_i \leq t} (1 - \overline{\bm{z}}_i ) > 0 \quad\text{ if and only if }\quad \sum_{t_i \leq t} \overline{\bm{{z}}}_i < + \infty.
\end{equation*}
Denote by $\tau_\epsilon := \inf\{t \geq 0\;:\; N_t \leq \epsilon\}$. 
Then, for any $K,\epsilon > 0$, we have
\begingroup
\allowdisplaybreaks
\begin{align*}
	\mathbb{P} \Bigg( \sum_{t_i \leq t} \overline{\bm{z}}_i < K \Bigg) &\geq \mathbb{P} \Bigg( \Big\{ \sum_{t_i \leq t} \overline{\bm{z}}_i < K \Big\} \cap \{ \tau_\epsilon > t \} \Bigg) \\
	&\geq \mathbb{P} \Bigg( \Big\{ \sum_{t_i \leq t} \frac{{z}_b^i}{(1 - {z}_d^i) \epsilon + {z}_b^i} < K \Big\} \cap \{ \tau_\epsilon > t \} \Bigg) \\
	&= 1-  \mathbb{P} \Bigg( \Big\{ \sum_{t_i \leq t} \frac{{z}_b^i}{(1 - {z}_d^i) \epsilon + {z}_b^i} \geq K \Big\} \cup \{ \tau_\epsilon \leq t \} \Bigg) \\
	&\geq 1-  \mathbb{P} \Bigg( \Big\{ \sum_{t_i \leq t} \frac{{z}_b^i}{(1 - {z}_d^i) \epsilon + {z}_b^i} \geq K \Big\} \Bigg) - \mathbb{P} ( \tau_\epsilon \leq t ) \\
	&\geq 1-  K^{-1} \mathbb{E} \Bigg[  \sum_{t_i \leq t} \frac{{z}_b^i}{(1 - {z}_d^i) \epsilon + {z}_b^i} \Bigg] - \mathbb{P} ( \tau_\epsilon \leq t )\\
	&= 1 - K^{-1} t\int_\mathcal{Z} \dfrac{{z}_b}{(1-{z}_d)\epsilon + {z}_b}\d\Pi(\bm{z}) - \mathbb{P}(\tau_\epsilon \leq t).
\end{align*}
\endgroup
The integral in the last line is finite as
\[
\int_\mathcal{Z} \dfrac{{z}_b}{(1-{z}_d)\epsilon + {z}_b}\d\Pi(\bm{z}) \leq \dfrac{2}{\epsilon}\int_{\{{z}_d\leq 1/2\}} {z}_d\d\Pi(\bm{z}) + \int_{\{{z}_d\geq 1/2\}} \d\Pi(\bm{z}) < +\infty.
\]
If we choose $K$ large enough and $\epsilon$ small enough, the above is bounded away from zero and the statement follows.
\end{proof}


\section*{Acknowledgements}
The authors wish to thank Alison Etheridge and David Helekal for valuable suggestions and discussions. 
The hospitality of the Hausdorff Institute for Mathematics Bonn in the course of the 2022 Junior Trimester Program "Stochastic modelling in life sciences" is gratefully remembered by the authors. 
This program was funded by the Deutsche Forschungsgemeinschaft (DFG, German Research Foundation) under Germany's Excellence Strategy - EXC-2047/1 - 390685813.
J.K. has been partially funded by Deutsche Forschungsgemeinschaft (DFG) through grant CRC 1114 ``Scaling Cascades in Complex Systems", Project Number 235221301, Project C02 ``Interface dynamics: Bridging stochastic and hydrodynamic descriptions".
B.W. was supported by the UK Engineering and Physical Sciences Research Council Grant [EP/V520202/1].

\nomenclature{$\alpha$}{Projection to extended space in lookdown}
\nomenclature{$\Xi$}{De Finetti measure of the lookdown}
\nomenclature{$D$}{Fr\'echet derivative}
\nomenclature{$\mathbb{D}$}{C\'adl\'ag functions}
\nomenclature{$\overleftarrow{N}$}{Time reversal of the population size}
\nomenclature{$C_b^k$}{$k$-times countinuously bounded differentiable functions}
\nomenclature{$\mathbb{K}$}{Compact type tpace, e.g.~$[0,1]$}
\nomenclature{$N$}{Total population size in $(0,+\infty)$}
\nomenclature{$\rho$}{Distribution of types on $\mathbb{K}$}
\nomenclature{$\sigma$}{Unnormalized type measure on $\mathbb{K}$}
\nomenclature{$\xi$}{Poisson process driving a stochastic process}
\nomenclature{$\lambda$}{Maximum level in the continuous lookdown}
\nomenclature{$\mu$}{Distribution of the population size $N$}
\nomenclature{$\Pi$}{Jump measure on the set of events $\mathcal{Z}$}
\nomenclature{$\delta_x$}{Dirac-Delta at $x$ is the unit atom at $x$}
\nomenclature{$\mathcal{Z}$}{The event space $\big( [0,1]\times [0,+\infty)\big)\setminus\{(0,0)\}$}
\nomenclature{$\bm{z}$}{An event $\bm{z} = ({z}_d,{z}_b)\in \mathcal{Z}$.}
\nomenclature{$\overline{\bm{{z}}}$}{Effective impact ${z}_b / ((1 - {z}_d) N + {z}_b)$}
\nomenclature{$\bm\gamma$}{Continuous part $\bm\gamma = (\gamma_d,\gamma_b)$ of the process.}
\nomenclature{$S$}{State space $(0,+\infty)\times \mathcal{M}_1 (\mathbb{K})$ of the $(\bm{\gamma},\Pi)$-Fleming-Viot process}
\nomenclature{$\mathcal{N}$}{Set of couting measures on $[0, +\infty)$}
\nomenclature{$\mathcal{M}_F$}{Space of finite measures}
\nomenclature{$\mathcal{M}_1$}{Space of probability measures}
\nomenclature{$\lint 1, n \rint$}{The set $\{ 1,..., n\}$ }
\nomenclature{$\mathbb{N}_0$}{The set of natural numbers $\{0,1,2,\dots\}$}
\nomenclature{$\mathbb{N}$}{The set of non-zero natural numbers $\mathbb{N}\setminus \{0\}$}
\nomenclature{$\mathrm{Exp}(m)$}{Exponential distribution of mean $m$}
\nomenclature{$\mathrm{Geo}(m)$}{Geometric distribution of mean $m$}
\nomenclature{$\mathrm{Pois}(m)$}{Poisson distribution of mean $m$}


%
\printbibliography

\end{document}